\documentclass[12pt]{amsart}
\pdfoutput=1
\usepackage{latexsym, amsmath, amssymb, amsthm, bm, textcase}
\usepackage{tikz}
\usetikzlibrary{arrows,shapes,positioning,calc,patterns}
\usepackage[mathscr]{eucal}
\usepackage[alwaysadjust]{paralist}
\usepackage[colorlinks=true,linkcolor=blue,urlcolor=blue, citecolor=blue]%
  {hyperref}
\usepackage[shortalphabetic]{amsrefs}
\usepackage[T1]{fontenc}
\usepackage{mathptmx}
\usepackage{microtype}
\usepackage{verbatim}

\usepackage[centering, includeheadfoot, hmargin=1in, vmargin=1in,
  headheight=30.4pt]{geometry}
\newtheorem{lemma}{Lemma}[section]
\newtheorem{theorem}[lemma]{Theorem}
\newtheorem{corollary}[lemma]{Corollary}

\newtheorem{proposition}[lemma]{Proposition}
\newtheorem*{proposition1}{Proposition~\ref{pro:sections}}
\newtheorem*{theorem2}{Theorem~\ref{thm:main1}}

\theoremstyle{definition}
\newtheorem{algorithm}[lemma]{Algorithm}
\newtheorem{remark}[lemma]{Remark}
\newtheorem{example}[lemma]{Example}
\newtheorem*{future}{Future directions}

\newtheorem*{conventions}{Conventions}
\newtheorem*{acknowledgements}{Acknowledgements}

\newcommand{\define}[1]{{\bfseries\itshape #1}}
\newcommand{\relphantom}[1]{\mathrel{\phantom{#1}}}
 
\newcommand{\CC}{\ensuremath{\mathbb{C}}} 
\newcommand{\NN}{\ensuremath{\mathbb{N}}}
\newcommand{\PP}{\ensuremath{\mathbb{P}}} 
 
\newcommand{\RR}{\ensuremath{\mathbb{R}}} 
\newcommand{\ZZ}{\ensuremath{\mathbb{Z}}} 
\newcommand{\cE}{\ensuremath{\mathcal{E}}} 
\newcommand{\cF}{\ensuremath{\mathcal{F}}} 
\newcommand{\cG}{\ensuremath{\mathcal{G}}} 
\newcommand{\cH}{\ensuremath{\mathcal{H}}} 
\newcommand{\cL}{\ensuremath{\mathcal{L}}} 
\newcommand{\cO}{\ensuremath{\mathcal{O}}} 
\newcommand{\cQ}{\ensuremath{\mathcal{Q}}} 
\newcommand{\cT}{\ensuremath{\mathcal{T}}} 
\newcommand{\sB}{\ensuremath{\mathscr{B}}} 

\newcommand{\ee}{\ensuremath{\mathbf{e}}} 
\newcommand{\uu}{\ensuremath{\mathbf{u}}} 
\newcommand{\vv}{\ensuremath{\mathbf{v}}} 
\newcommand{\ww}{\ensuremath{\mathbf{w}}} 
\renewcommand{\geq}{\geqslant}
\renewcommand{\leq}{\leqslant}
\newcommand{\into}{\hookrightarrow}
\renewcommand{\div}{\mydiv}
\DeclareMathOperator{\conv}{conv}
\DeclareMathOperator{\Div}{Div}
\DeclareMathOperator{\Mat}{M}

\DeclareMathOperator{\mydiv}{div}
\DeclareMathOperator{\Hom}{Hom}
\DeclareMathOperator{\Poset}{L}
\DeclareMathOperator{\pos}{pos}
\DeclareMathOperator{\Pic}{Pic}
\DeclareMathOperator{\rank}{rank}
\DeclareMathOperator{\Span}{span}
\DeclareMathOperator{\Spec}{Spec}
\DeclareMathOperator{\Sym}{Sym}
\DeclareMathOperator{\variety}{V}

\begin{document}

\vspace*{-3.0em}

\title{Toric vector bundles and parliaments of polytopes}

\author[S.~Di~Rocco]{Sandra Di Rocco} \address{Sandra Di Rocco\\ Department of
  Mathematics\\ Royal Institute of Technology (KTH)\\ 10044 Stockholm\\
  Sweden \\ 
  \href{mailto:dirocco@math.kth.se}%
  {{\ttfamily\upshape dirocco@math.kth.se}}}

\author[K.~Jabbusch]{Kelly Jabbusch}
\address{Kelly Jabbusch\\ Department of Mathematics \& Statistics\\ University
  of Michigan--Dearborn \\ 4901 Evergreen Road \\ Dearborn, Michigan
  48128-2406\\ USA \\ 
  \href{mailto:jabbusch@umich.edu}%
  {{\ttfamily\upshape jabbusch@umich.edu}}}

\author[G.G.~Smith]{Gregory G. Smith}
\address{Gregory G. Smith\\ Department of Mathematics and Statistics \\
  Queen's University\\ Kingston \\ Ontario \\ K7L~3N6\\ Canada \\
  \href{mailto:ggsmith@mast.queensu.ca}%
  {{\ttfamily\upshape ggsmith@mast.queensu.ca}}}

\subjclass[2010]{14M25; 14J60, 51M20}


\begin{abstract}
  We introduce a collection of convex polytopes associated to a
  torus-equivariant vector bundle on a smooth complete toric variety. We show
  that the lattice points in these polytopes correspond to generators for the
  space of global sections and we relate edges to jets.  Using the polytopes,
  we also exhibit toric vector bundles that are ample but not globally
  generated, and toric vector bundles that are ample and globally generated
  but not very ample.
\end{abstract}

\maketitle

\section{Overview of Results}
\label{sec:intro}

\noindent
The importance and prevalence of toric varieties stems from their
calculability and their close relation to polyhedral objects.  The challenge
is to emulate this success and enlarge the class of varieties with both
features.  Rather than contemplating spherical varieties or all $T$-varieties,
we extend the theory of toric varieties by studying torus-equivariant vector
bundles and their projective bundles.  Motivated by the ensuing
simplifications in the toric dictionary between line bundles and polyhedra, we
concentrate on vector bundles over a smooth complete toric variety.  The goal
of this paper is to give explicit polyhedral interpretations for properties of
these vector bundles.

To accomplish this goal, we fix a smooth complete toric variety $X$, over
$\CC$, associated to the fan $\Sigma$.  Let $M$ denote the character lattice
of the dense torus in $X$ and write
$\vv_1, \vv_2, \dotsc, \vv_n \in \Hom_\ZZ(M,\ZZ)$ for the unique minimal
generators of the rays in $\Sigma$.  A \define{toric vector bundle} on $X$ is
a torus-equivariant locally-free $\cO_X$-module $\cE$ of finite rank $r$.  The
celebrated Klyachko classification proves that $\cE$ corresponds to a
finite-dimensional vector space $E \cong \CC^r$ equipped with compatible
decreasing filtrations
$E \supseteq \dotsb \supseteq E^{i}(j) \supseteq E^{i}(j+1) \supseteq \dotsb
\supseteq 0$ where $1 \leq i \leq n$ and $j \in \ZZ$; see
Section~\ref{sec:bundles}.  This collection of linear subspaces embeds into
the lattice of flats for a distinguished matroid $\Mat(\cE)$.  For each
element $\ee$ in the ground set of the matroid $\Mat(\cE)$, we introduce the
convex polytope
\[
  P_\ee := \bigl\{ \uu \in M \otimes_\ZZ \RR :
\text{$\langle \uu, \vv_i \rangle \leq \max \bigl( j \in \ZZ : \ee \in
  E^{i}(j) \bigr)$ for all $1 \leq i \leq n$} \bigr\} \, .
\]  
The set of all such polytopes $P_\ee$ is called the \define{parliament of
  polytopes} for $\cE$; see Section~\ref{sec:globalSections}.  Although the
defining half-spaces for the polytopes $P_\ee$ together with the elements
$\ee$ in the ground set of $\Mat(\cE)$ encode the filtrations, the polytopes
themself may be empty; compare with Remark~\ref{rem:gens}.

The following result gives the first substantive connection between the
parliament of polytopes and the toric vector bundle.

\begin{proposition}
  \label{pro:sections}
  The lattice points in the polytopes of the parliament for $\cE$ correspond
  to the torus-equivariant generators for the space of global sections of
  $\cE$.
\end{proposition}

\noindent
Example~\ref{exa:lineBundle} recovers the polytope associated to a toric line
bundle on $X$.  However, when the rank of $\cE$ is greater than $1$,
Example~\ref{exa:tangentBundle} demonstrates that the lattice points in the
polytopes of the parliament need not yield a basis for the space of global
sections.  This highlights the key difference between higher-rank toric vector
bundles and toric line bundles: toric vector bundles depend on both the
combinatorics of the polytopes $P_\ee$ and the properties of the elements
$\ee$ in the ground set of the matroid $\Mat(\cE)$.  For line bundles, we may
overlook the elements indexing the polytope because linear algebra in a
one-dimensional vector space is trivial.  Our criterion for deciding whether a
toric vector bundle is globally generated underscores this distinction.

To outline this criterion, consider a maximal cone $\sigma \in \Sigma$.  The
restriction of the toric vector bundle $\cE$ to the affine open toric variety
$U_\sigma$ splits equivariantly as a direct sum of toric line bundles.  Since
toric line bundles on $U_\sigma$ correspond to lattice points in $M$, we
obtain a multiset $\bm{u}(\sigma) \subset M$ of associated characters for each
maximal cone $\sigma \in \Sigma$; see Section~\ref{sec:bundles}.  With this
notation, we have our second result.

\begin{theorem}
  \label{thm:main1}
  A toric vector bundle is globally generated if and only if, for all maximal
  cones $\sigma \in \Sigma$, the associated characters in $\bm{u}(\sigma)$ are
  vertices of polytopes in the parliament and the elements indexing these
  polytopes form a basis in the matroid $\Mat(\cE)$.
\end{theorem}

\noindent
Example~\ref{exa:nonvanishing} demonstrates that global generation is not
simply a property of the individual polytopes in the parliament, and
Example~\ref{exa:nonvanishing2} shows that the higher-cohomology groups of a
globally-generated ample toric vector bundle may be nonzero.

The parliament of polytopes for $\cE$ gives new insights into the projective
bundle $\PP(\cE)$.  This is particularly relevant for the positivity
properties of $\cE$ defined by the corresponding attribute for the
tautological line bundle $\cO_{\PP(\cE)}(1)$.  For instance, we may picture
the restriction of $\cE$ to a torus-invariant curve in $X$ as the normalized
distances between appropriately matched characters associated to $\cE$; see
Section~\ref{sec:basepointfree}.  Hence, Theorem~2.1 in \cite{HMP} allows us
to quickly recognize ample and nef toric vector bundles.  Exploiting our
polyhedral interpretations, Example~\ref{exa:notBPF2} exhibits a toric vector
bundle $\cF$ on $\PP^2$ that is ample but not globally generated, and
Example~\ref{exa:span} exhibits a toric vector bundle $\cH$ on $\PP^2$ that is
ample and globally generated but not very ample.  Better still,
Proposition~\ref{pro:lowRank} and Remark~\ref{rem:lowRank} prove that $\cF$
and $\cH$ have the minimal rank among all toric vector bundles on $\PP^d$ with
the given traits.  Beyond answering Question~7.5 in \cite{HMP}, these examples
reinforce the conventional wisdom that versions of positivity that coincide
for line bundles diverge for higher-rank vector bundles.

The discrete geometry within the parliament of polytopes nevertheless captures
the positivity of jets.  In contrast with the conventional wisdom, several
forms of higher-order positivity are equivalent for toric vector bundles.  A
vector bundle $\cE$ separates $\ell$-jets for $\ell \in \NN$ if, for every
closed point $x \in X$ with maximal ideal $\mathfrak{m}_x \subseteq \cO_X$,
the natural map
$H^0(X, \cE) \to H^0(X, \cE \otimes_{\cO_X} \cO_X / \mathfrak{m}_x^{\ell+1})$
is surjective; see Section~\ref{sec:jets}.  As an enhancement of
Theorem~\ref{thm:main1}, Theorem~\ref{thm:kjet} establishes that a toric
vector bundle $\cE$ separates $\ell$-jets if and only if certain edges in the
polytopes of the parliament have normalized length at least $\ell$.  This
leads to the following equivalences.

\begin{theorem}
  \label{thm:main?}
  A toric vector bundle $\cE$ separates $\ell$-jets if and only if it is
  $\ell$-jet ample.  Moreover, a toric vector bundle $\cE$ separates $1$-jets
  if and only if it is very ample.
\end{theorem}

\noindent
Unlike arbitrary vector bundles on a smooth projective variety, these versions
of positivity coincide for toric vector bundles.  Specializing to line
bundles, we recover the main theorems in \cite{DiRocco}.  We also obtain a
polyhedral characterization for very ampleness; see Corollary~\ref{cor:very}.

\begin{future}
  The introduction of the parliament of polytopes for a toric vector bundle
  suggests some new research projects.  The most straightforward advances
  would provide polyhedral interpretations for other properties of toric
  vector bundles.  For example, we suspect that a toric vector bundle is big
  if and only if some Minkowski sum of the polytopes in the parliament is
  full-dimensional.  For a globally-generated toric vector bundle $\cE$, the
  complete linear series of $\cO_{\PP(\cE)}(1)$ maps the projective bundle
  $\PP(\cE)$ into projective space.  Can one characterize the homogeneous
  equations of the image in terms of combinatorial commutative algebra?  If
  so, then one expects a description of the initial ideals via regular
  triangulations; compare with Section~8 in \cite{Sturmfels}.  Since there
  exists ample, but not globally generated, line bundles on varieties of the
  form $\PP(\cE)$, this class of varieties makes an interesting testing ground
  for Fujita's conjecture; see Conjecture~10.4.1 in \cite{PAG2}.  More
  ambitiously, for an ample toric vector bundle $\cE$, one could even ask for
  an effective polyhedral bound on $m \in \NN$ such that $\Sym^m(\cE)$ is
  globally generated or very ample.  Finally, we wonder if there are natural
  topological hypotheses on the parliament of polytopes which imply that all
  of the higher-cohomology groups vanish.
\end{future}

\begin{conventions}
  Throughout the document, $\NN$ denotes the nonnegative integers and $X$ is a
  smooth complete toric variety over the complex numbers $\CC$.  The linear
  subspace generated by the vectors $\ee_1, \ee_2, \dotsc, \ee_m$ in a
  $\CC$\nobreakdash-vector space is denoted by
  $\Span(\ee_1,\ee_2, \dotsc, \ee_m)$, and the polyhedral cone generated by
  the vectors $\vv_1, \vv_2, \dotsc, \vv_m$ in a $\RR$-vector space is denoted
  by $\pos(\vv_1,\vv_2, \dotsc, \vv_m)$.
\end{conventions}

\begin{acknowledgements}
  We thank Alex Fink, Milena Hering, Nathan Ilten, Diane Maclagan, Bernt Ivar
  Utst{\o}l N{\o}dland, Sam Payne, Vic Reiner, Mike Roth, and Frank Sottile
  for helpful conversations.  We especially thank an anonymous referee for
  wonderfully constructive feedback and for suggesting Example~\ref{exa:Ref}.
  The first author was partially supported by the Vetenskapsr\aa det grants
  NT:2010\nobreakdash-5563 and NT:2014\nobreakdash-4736, the second was
  partially supported by the G\"{o}ran Gustafsson Stiftelse, and the third was
  partially supported by NSERC.
\end{acknowledgements}

\section{Background on Toric Vector Bundles}
\label{sec:bundles}

\noindent
In this section, we collect the needed definitions and notation for toric
varieties and vector bundles.

Let $X$ be a smooth complete $d$-dimensional toric variety, over $\CC$,
determined by the strongly convex rational polyhedral fan $\Sigma$ in
$N \otimes_\ZZ \RR \cong \RR^d$, where $N$ is a lattice of rank $d$.  The dual
lattice is $M := \Hom_{\ZZ}(N,\ZZ)$, and the dense algebraic torus acting on
$X$ is $T := \Spec \CC[M]$.  For $\sigma \in \Sigma$, the corresponding affine
toric variety is $U_\sigma := \Spec \CC[\sigma^\vee \cap M]$, where
$\sigma^\vee$ denotes the dual cone. The $j$-dimensional cones of $\Sigma$
form the set $\Sigma(j)$.  For each maximal cone $\sigma \in \Sigma(d)$, the
corresponding $T$-fixed point is $x_\sigma \in X$.  We order the
$1$-dimensional cones $\Sigma(1)$ (also known as rays) and, for
$1 \leq i \leq n$, we write $\vv_i \in N$ for the unique minimal generator of
the $i$-th ray.  The $i$-th ray also corresponds to the irreducible
$T$-invariant divisor $D_i$ on $X$, and the divisors $D_1, D_2, \dotsc, D_n$
generate the group $\Div_T(X) \cong \ZZ^n$ of $T$-invariant divisors.  Since
$X$ is complete, there is a short exact sequence
\[
  0 \longrightarrow M \xrightarrow{\; \div \;} \Div_T(X) \longrightarrow
\Pic(X) \longrightarrow 0
\] 
where
$\div \uu := \langle \uu, \vv_1 \rangle D_1 + \langle \uu, \vv_2 \rangle D_2 +
\dotsb + \langle \uu, \vv_n \rangle D_n$ and the second map is the projection
from the group of divisors to the Picard group.  The invertible sheaf or line
bundle associated to a divisor $D \in \Div_T(X)$ is denoted by $\cO_X(D)$.
For more information on toric varieties, see \cite{CLS} or \cite{Fulton}.

A \define{toric vector bundle} is a locally-free $\cO_X$-module $\cE$ of
finite rank $r$ equipped with a $T$-action that is compatible with the
$T$-action on $X$. In other words, there exists a $T$-action on the variety
$\variety(\cE) := \Spec ( \Sym \cE)$ such that the projection map
$\pi \colon \variety(\cE) \to X$ is $T$-equivariant and $T$ acts linearly on
the fibres.  For all $\sigma \in \Sigma$, there is also an induced $T$-action
on the $\CC$-vector spaces of sections $H^0(U_\sigma,\cE)$, where $U_\sigma$
is the corresponding affine toric variety.  Given a lattice point $\uu \in M$,
the trivial line bundle $\cO_X( \div \uu )$ has a canonical $T$-equivariant
structure.  Explicitly, for all $\sigma \in \Sigma$, we have
\[
  H^0 \bigl(U_\sigma, \cO_X( \div \uu ) \bigr) = \bigoplus_{\uu' \in
    \sigma^\vee \cap M} \CC \cdot \chi^{\uu' - \uu} \subset T \, ,
\] 
where $\chi^{\uu'}, \chi^{\uu}$ are the characters associated to the lattice
points $\uu', \uu \in M$; the identity in this semigroup is $\chi^{-\uu}$.  As
in \cite{HMP}, we follow the standard convention in invariant theory for the
action of the group on the ring of functions, even though the opposite sign
convention is more common in the toric literature.

Every toric line bundle on the affine toric variety $U_\sigma$ is
$T$-equivariantly isomorphic to a line bundle $\cO_X( \div \uu )|_{U_\sigma}$,
where the class $\overline{\uu}$ of the lattice point $\uu$ in
$M_\sigma := M / (\sigma^{\perp} \cap M)$ is uniquely determined.  In
addition, any toric vector bundle on an affine toric variety splits
$T$-equivariantly as a direct sum of toric line bundles whose underlying line
bundles are trivial; see Proposition~2.2 in \cite{Payne}.  Hence, for all
$\sigma \in \Sigma$, there is a unique multiset
$\bm{u}(\sigma) \subset M_\sigma$ such that
$\cE|_{U_\sigma} \cong \bigoplus_{\overline{\uu} \in \bm{u}(\sigma)} \cO_X(
\div \uu ) |_{U_\sigma}$, where $\uu \in M$ is any lift of $\overline{\uu}$.
If $\sigma$ is a maximal cone, then the multiset $\bm{u}(\sigma) \subset M$ is
uniquely determined by the toric vector bundle $\cE$ and the $d$-dimensional
cone $\sigma$.  We call the multisets $\bm{u}(\sigma)$, for all
$\sigma \in \Sigma(d)$, the \emph{associated characters} of the toric vector
bundle $\cE$.

Toric vector bundles are classified in Theorem~0.1.1 of \cite{Kly} by
canonical filtrations.  To summarize this classification, let $E$ be the fibre
of $\cE$ over the identity of the torus $T$, so $E$ is a $\CC$-vector space
isomorphic to $\CC^r$.  The category of toric vector bundles on $X$ is
naturally equivalent to the category of finite-dimensional $\CC$-vector spaces
$E$ with separated exhaustive decreasing filtrations
$\{ E^{i}(j) \}_{j \in \ZZ}$, for all $1 \leq i \leq n$, that satisfy the
compatibility condition: \setcounter{equation}{0}
\begin{equation}
  \label{eq:compatibility}
  \tag{$\text{CC}$}
  \text{
    \begin{tabular}{p{11cm}}
      For each maximal cone $\sigma \in \Sigma(d)$, there exists a decomposition
      $E = \displaystyle\bigoplus\nolimits_{\uu \in \bm{u}(\sigma)} L_{\uu}$ such
      that $E^{i}(j) = \displaystyle\sum\nolimits_{\langle \uu, \vv_i
      \rangle \geq j} L_\uu$.
    \end{tabular}}
  \setcounter{equation}{1}
\end{equation}
This compatibility condition is equivalent to the $T$-equivariant splitting
into a direct sum of toric line bundles on the affine open toric variety
$U_\sigma$, for all $\sigma \in \Sigma(d)$; see Theorem~1.3.2 in \cite{Kly2}.
Indirectly, the decreasing filtrations provide the gluing data needed to
assemble these direct sums into a toric vector bundle.  The filtrations being
separated and exhaustive, for each $1 \leq i \leq n$, means that
$E^{i}(j) = 0$ for all $j \gg 0$ and $E^{i}(j) = E \cong \CC^r$ for all
$j \ll 0$, so each filtration contains only finitely many distinct linear
subspaces.  Hence, for a fixed $i$, we may conveniently describe the
filtration $\{ E^{i}(j) \}_{j \in \ZZ}$ via a labelled basis
$\ee_1, \ee_2, \dotsc, \ee_r$ for $E \cong \CC^r$, where each vector
$\ee_k \in E$ is labelled by an integer and the linear subspace $E^{i}(j)$ is
simply the span of the basis vectors with labels greater than or equal to $j$.
For a self-contained exposition of this classification, we recommend
Subsection~2.3 in \cite{Payne}; Subsection~2.4 in \cite{Payne} also provides a
brief historical summary.

Given a toric vector bundle $\cE$, the filtrations
$\{ E^{i}(j) \}_{j \in \ZZ}$ have a couple different geometric
interpretations.  For all cones $\sigma \in \Sigma$ and all lattice points
$\uu \in M$, evaluating sections at the identity of the torus $T$ gives an
injective map $H^0(U_\sigma, \cE)_\uu \into E$.  The image of this map is the
linear subspace $E_{\uu}^\sigma \subseteq E$.  Following Subsection~4.2 in
\cite{Payne2}, we define a linear subspace $E^{\vv}(j) \subseteq E$ for all
$\vv \in N$ and all $j \in \ZZ$.  Since $X$ is complete, there exists a unique
cone $\sigma \in \Sigma$ containing the lattice point $\vv$ in its relative
interior.  Set
$E^{\vv}(j) := \sum_{\langle \uu, \vv \rangle \geq j} E_{\uu}^\sigma$.  For
any lattice point $\vv \in N$, the family of linear subspaces
$\{ E^{\vv}(j) \}_{j \in \ZZ}$ give a separated exhaustive decreasing
filtration of $E$.  When the lattice point $\vv$ equals $\vv_i$ for some
$1 \leq i \leq n$, we obtain the filtration $\{ E^{i}(j) \}_{j \in \ZZ}$.

For the second interpretation of the filtrations, consider a cone
$\sigma \in \Sigma$ and suppose that we have
$\cE|_{U_\sigma} \cong \bigoplus_{\overline{\uu} \in \bm{u}(\sigma)} \cO_X(
\div \uu ) |_{U_\sigma}$.  If the linear subspace $L_{\uu} \subseteq E$ is the
fibre of $\cO_X( \div \uu )$ over the identity of the torus $T$, then we
obtain a decomposition
$E = \bigoplus_{\overline{\uu} \in \bm{u}(\sigma)} L_{\uu}$.  Hence, the
linear subspace $E_{\uu'}^\sigma$ is spanned by the linear subspaces $L_{\uu}$
for which $\uu - \uu' \in \sigma^\vee$ and
$E^{\vv}(j) = \bigoplus_{\langle \uu, \vv \rangle \geq j} L_{\uu}$.  For each
maximal cone $\sigma \in \Sigma$, there exists a subset
$\underline{\bm{u}}(\sigma) \subset M$ and a decomposition
$E = \bigoplus_{\uu \in \underline{\bm{u}}(\sigma)} E_{\uu}$ such that, for
all $\vv \in \sigma$ and for all $j \in \ZZ$, we have
$E^{\vv}(j) = \bigoplus_{\langle \uu, \vv \rangle \geq j} E_{\uu}$.  It
follows that $E_{\uu} = \bigoplus_{\uu \in \bm{u}(\sigma)} L_{\uu}$, so
$\dim E_{\uu}$ equals the multiplicity of $\uu$ in the multiset
$\bm{u}(\sigma)$ and $\underline{\bm{u}}(\sigma)$ is the underlying set of
elements in $\bm{u}(\sigma)$.

\section{Global Sections and Lattice Polytopes}
\label{sec:globalSections}

\noindent
This section introduces explicit $T$-equivariant generators for the global
sections of the toric vector bundle that correspond to the lattice points in a
collection of polytopes.  Each toric line bundle $\cL$ on $X$ corresponds to a
rational convex polytope in $M \otimes_\ZZ \RR$.  We generalize this
correspondence by associating a finite collection of convex polytopes to a
toric vector bundle $\cE$.  The polytopes in this collection are indexed by
the elements in the ground set of a matroid associated to $\cE$.

To describe this matroid, we first observe that the toric vector bundle $\cE$
determines the finite poset $\Poset(\cE)$, consisting of all the linear
subspaces $V := \bigcap_{i=1}^n E^{i}(j_i) \subseteq E$, where
$(j_1, j_2, \dotsc, j_n) \in \ZZ^n$, ordered by inclusion.  Since the
filtrations $\{ E^{i}(j) \}_{j \in \ZZ}$ are separated, exhaustive, and
decreasing, we see that $0 \in \Poset(\cE)$, $E \in \Poset(\cE)$, and
$\Poset(\cE)$ is closed under intersection.  Hence, the pair
$\bigl( \Poset(\cE), \cap \bigr)$ forms a meet-semilattice.  The next result
shows that $\Poset(\cE)$ embeds into the lattice of flats for a distinguished
representable matroid.

\begin{proposition}
  \label{pro:matroid}
  For a toric vector bundle $\cE$, there exists a unique matroid $\Mat(\cE)$,
  representable over $\CC$, such that
  \begin{enumerate}[\bfseries\upshape (M1)]
  \item the poset $\Poset(\cE)$ is isomorphic to a meet-subsemilattice in the
    lattice of flats,
  \item among all matroids satisfying \emph{(M1)}, the number of elements in the
    ground set is minimal, and
  \item among all matroids satisfying \emph{(M1)} and \emph{(M2)}, the number
    of circuits is minimal.
  \end{enumerate}
\end{proposition}

\noindent
In the language of linear subspace arrangements (and ordering the subspaces by
reversed inclusion), Proposition~\ref{pro:matroid} is equivalent to
Theorem~I.4.9 in \cite{Ziegler}.

\begin{proof}
  We verify that Algorithm~\ref{alg:matroid} returns a representable matroid
  $\Mat$ with the desired conditions. 
  \begin{table}[ht]
    \begin{algorithm}[Construction of the representable matroid associated to a
      toric vector bundle] 
      \label{alg:matroid}
      $\;$
    
      \begin{tabbing}
        \= Output: \= \kill
        \> Input: \> The poset $\Poset(\cE)$ of linear subspaces
        associated to the toric vector bundle $\cE$. \\
        \> Output: \> The canonical matroid $\Mat(\cE)$ associated to
        $\cE$. \\[0.35em]
        \= W \= W \= W \= W \= \kill
        \> Set $r$ to be the dimension of the largest linear subspace $E$ in
        $\Poset(\cE)$;\\
        \> Initialize $G$ to be a set consisting of a basis vector for each
        one-dimensional subspace in $\Poset(\cE)$; \\
        \> For each integer $k$ from $2$ to $r$ do \\
        \> \> For each $k$-dimensional linear subspace $V$ in $\Poset(\cE)$ do\\
        \> \> \> Set $G'$ to be the subset of elements in $G$ that lie in
        $V$;\\
        \> \> \> If the linear subspace $\Span(G')$ is a proper subspace in
        $V$ then\\
        \> \> \> \>  Append to $G$ a basis for a complementary subspace to
        $\Span(G')$ in $V$; \\
        \> Return the linear matroid defined by the vectors in $G$.
      \end{tabbing}
    \end{algorithm}
  \end{table}
  By construction, each linear subspace $V$ in $\Poset(\cE)$ is generated by a
  subset of vectors in the ground set of the matroid $\Mat$.  The subset of
  the ground set consisting of all elements contained in $V$ is the flat $F_V$
  in $\Mat$ corresponding to $V$.  It follows that $\Span(F_V) = V$,
  $\rank(F_V) = \dim(V)$, and the induced injective map from the poset
  $\Poset(\cE)$ into lattice of flats for $\Mat$ is compatible with
  intersections.  Thus, the matroid $\Mat$ satisfies the condition~(M1).

  For any matroid, the lattice of flats is relatively complemented; see
  Proposition~3.4.4 in \cite{matroid}.  It follows that, for any linear
  subspace $V$ in $\Poset(\cE)$ and any matroid satisfying condition~(M1),
  there exists a flat $F'$ such that the join of $F_V$ and $F'$ is $F_E$ and
  the meet of $F_V$ and $F'$ is $F_{\{0\}}$.  By iterating from the smallest
  to the largest linear subspaces in $\Poset(\cE)$, the
  Algorithm~\ref{alg:matroid} finds a minimal set of complementary subspaces
  for $\Poset(\cE)$.  Adjoining these to $\Poset(\cE)$, we obtain a new
  meet-semilattice $L'$ such that the complementary subspaces are minimal
  among the nonzero subspaces, and every linear subspace is generated by some
  collection of minimal nonzero subspaces.  Using the terminology from
  Section~3.4 in \cite{matroid}, we see that the atoms in $L'$ are the
  one-dimension linear subspaces in $\Poset(\cE)$ together with the adjoined
  complementary subspaces.  Moreover, $L'$ is the minimal atomistic
  meet-semilattice containing $\Poset(\cE)$.

  Finally, we claim that the matroid $\Mat$ is the free expansion of $L'$; see
  Proposition~10.2.3 in \cite{matroid}.  By construction, the ground set of
  $\Mat$ consists of a basis for each atom in $L'$, so the number of elements
  in the ground set of $\Mat$ equals the number of elements in the ground set
  of the free expansion of $L'$.  Moreover, the conditional statement in
  Algorithm~\ref{alg:matroid} implies that a flat $D$ in the matroid $\Mat$ is
  dependent if and only if there exists a linear subspace $W \in L'$ such that
  $|D \cap F_W| > \dim(W)$.  We conclude that $\Mat$ is the free expansion of
  $L'$.  Therefore, Proposition~10.2.2 and Proposition~10.2.6 in
  \cite{matroid} establish that the matroid $\Mat$ satisfies conditions (M2)
  and (M3) respectively.
\end{proof}

\begin{remark}  
  \label{rem:split}
  Since $E \in \Poset(\cE)$, Algorithm~\ref{alg:matroid} shows that the number
  of elements in the ground set of the matroid $\Mat(\cE)$ is at least the
  rank $r$ of $\cE$.  To have equality, there must be a basis for $E$ such
  that every linear subspace in $\Poset(\cE)$ is a direct sum of coordinate
  subspaces. Hence, the number of elements in the ground set of the matroid
  $\Mat(\cE)$ equals $r$ if and only if the toric vector bundle $\cE$ splits
  $T$-equivariantly into a direct sum of toric line bundles.
\end{remark}

For each maximal cone $\sigma \in \Sigma(d)$, the compatibility condition
\eqref{eq:compatibility} is equivalent to saying that the subposet of
$\Poset(\cE)$ consisting of the linear subspaces
$\bigcap_{\vv_i \in \sigma} E^{i}(j_i)$, where $j_i \in \ZZ$, is a
distributive lattice; see Remark~2.2.2 in \cite{Kly}.  Equivalently, the
matroid $\Mat(\cE)$ contains a \define{compatible basis $\sB_{\sigma}$} such
that each component $E^{i}(j)$, for $\vv_i \in \sigma$ and $j \in \ZZ$, is a
direct sum of the corresponding coordinate subspaces.
Example~\ref{exa:nonvanishing} demonstrates that, for a given maximal cone
$\sigma$, there may be more that basis in $\Mat(\cE)$ with this property.

\begin{remark}
  \label{rem:intrinsic}
  For an element $\ee$ in the ground set of the matroid $\Mat(\cE)$ and a
  linear subspace $V \in \Poset(\cE)$, the relation $\ee \in V$ depends only
  on the matroid $\Mat(\cE)$ and not on the choice of a representation for
  $\Mat(\cE)$.  Nevertheless, Algorithm~\ref{alg:matroid} does produce a
  particular representation for $\Mat(\cE)$.  This is analogous to a minimal
  free presentation for a finitely generated graded module over a polynomial
  ring: the ranks of the free modules are intrinsic invariants, but the matrix
  representing the map depends on the choice of bases; compare with Section~1B
  in \cite{Eis}.
\end{remark}

For each element $\ee$ in the ground set of the matroid $\Mat(\cE)$, the
associated convex polytope is
\[
P_{\ee} := \bigl\{ \uu \in M \otimes_\ZZ \RR :
\text{$\langle \uu, \vv_i \rangle \leq \max\bigl( j \in \ZZ : \ee \in
  E^{i}(j) \bigr)$ for all $1 \leq i \leq n$} \bigr\} \, .
\]  
Using a traditional term of venery (namely, the collective noun for owls), we
call the collection of all such polytopes $P_\ee$ the \define{parliament of
  polytopes} for the toric vector bundle $\cE$.  The number of polytopes in
the parliament for $\cE$ is at least the rank of $\cE$ and equals the rank of
$\cE$ precisely when $\cE$ splits into a direct sum of toric line bundles; see
Remark~\ref{rem:split}.

Extending the classic theorem~\cite{CLS}*{Theorem~4.3.3} for line bundles on a
toric variety, we have the following interpretation for the lattice points in
a parliament of polytopes.

\begin{proposition1}
  The lattice points in the polytopes of the parliament for $\cE$ correspond
  to the $T$-equivariant generators for the space of global sections of $\cE:$
  \[
    H^0(X, \cE) \cong \sum_{\ee} \Span ( \ee \otimes \chi^{-\uu} : \uu \in
    P_{\ee} \cap M ) \subset E \otimes_\CC T \, ,
  \]
  where the sum is over all elements $\ee$ in the ground set of the matroid
  $\Mat(\cE)$.
\end{proposition1}

\begin{proof}[Proof of Proposition~1.1]
  The $T$-action on the space of global sections yields a decomposition into
  isotypical components $H^0(X, \cE)_\uu$, where $\uu \in M$.  The regular
  $T$-eigenfunction $\chi^{- \uu}$ is an element of $H^0(X, \cE)_\uu$ and we
  have $H^0(X, \cE) = \bigoplus_{\uu \in M} H^0(X, \cE)_\uu$.  Since $X$ is
  complete, at most finitely many of the isotypical components are nonzero.
  Following Corollary~4.1.3 in \cite{Kly}, evaluation at the identity of the
  torus $T$ gives a canonical isomorphism
  \[
  H^0(X, \cE)_\uu = \bigcap_{\sigma \in \Sigma(d)} H^0(U_\sigma, \cE)_\uu
  \xrightarrow{\;\; \cong \;\;} \bigcap_{\sigma \in \Sigma(d)} E_\uu^\sigma 
  = \bigcap_{i=1}^{n}
  E^{i}(\langle \uu , \vv_i \rangle) \, .
  \]
  Since the linear subspace
  $V_{\uu} := \bigcap_{i=1}^{n} E^{i}(\langle \uu , \vv_i \rangle)$ belongs to
  the poset $\Poset(\cE)$, Proposition~\ref{pro:matroid} shows that there
  exists a flat $F$ in the matroid $\Mat(\cE)$ such that $\Span(F) = V_{\uu}$.
  Hence, we obtain the isomorphism
  $H^0(X,\cE)_{\uu} \cong \sum\limits_{\ee} \Span( \ee \otimes \chi^{-\uu} :
  \ee \in V_{\uu})$.  Because we have
  \begin{alignat*}{4}
    \ee &\in V_\uu &\quad &\Longleftrightarrow \quad &\ee &\in
    E^{i}(\langle \uu, \vv_i \rangle) && \text{for all
      $1 \leq i \leq n$} \\
    & &\quad &\Longleftrightarrow \quad &\langle \uu, \vv_i \rangle &\leq \max
    \bigl( j \in \ZZ : \ee \in E^{i}(j) \bigr) &\quad& \text{for all
      $1 \leq i \leq n$}  \\
    & &\quad &\Longleftrightarrow \quad & \uu &\in P_{\ee} \cap M \, , 
  \end{alignat*}
  we conclude that
  $H^0(X,\cE)_{\uu} \cong \sum\limits_{\ee} \Span( \ee \otimes \chi^{-\uu} :
  \uu \in P_{\ee} \cap M)$.
\end{proof}

As expected, we recover the description for the global sections of a line
bundle.

\begin{example}
  \label{exa:lineBundle}
  Every line bundle $\cL$ on a smooth toric variety $X$ equals $\cO_X(D)$ for
  some $T$-invariant divisor $D = a_1 D_1 + a_2 D_2 + \dotsb + a_n D_n$.
  Theorem~6.1.7 in \cite{CLS} establishes that the Cartier divisor $D$ is
  determined by a collection $\{ \uu_\sigma \in M : \sigma \in \Sigma(d) \}$,
  so we obtain $\bm{u}(\sigma) = \{ \uu_\sigma \}$ for all
  $\sigma \in \Sigma(d)$.  The associated continuous piecewise linear function
  $\varphi_D : N_\RR \to \RR$ satisfies $\varphi_D(\vv_i) = - a_i$ and
  $\varphi_D(\vv) = \langle \uu_\sigma, \vv \rangle $ for all
  $\vv \in \sigma$.  Following Subsection~2.3.1 in \cite{Kly}, the decreasing
  filtrations corresponding to $\cL$ are
  \[
  E^{i}(j) := 
  \left\{
    \renewcommand{\arraystretch}{0.9}
    \renewcommand{\arraycolsep}{1pt}
    \begin{array}{lcc}
      \CC & \;\; \text{if} \;\; & j \leq a_i \\
      0 & \text{if} & j > a_i
    \end{array}
  \right.
  \qquad \text{for all $1 \leq i \leq n$.}
  \] 
  If $\ee$ is any nonzero vector in $E = \CC$, then the ground set of the
  matroid $\Mat(\cL)$ is $\{ \ee \}$ and the unique polytope in the parliament
  is
  $P_\ee = \{ \uu \in M \otimes_\ZZ \RR : \langle \uu, \vv_i \rangle \leq a_i
  \}$.  It follows that $E_{\uu_\sigma} = E = \CC$ for all
  $\sigma \in \Sigma(d)$, so $H^0(X,\cL)_\uu = \CC$ when
  $\langle \uu, \vv_i \rangle \leq a_i$ for all $1 \leq i \leq n$ and
  $H^0(X,\cL)_\uu = 0$ otherwise.  Therefore, we have
  $H^0(X,\cL) = \bigoplus_{\uu \in P_\ee \cap M} \Span (\ee \otimes \chi^{-
    \uu})$.  Be aware that we use the opposite sign convention when compared
  to either Section~6.1 in \cite{CLS} or Section~3.4 in \cite{Fulton}.  \hfill
  $\Diamond$
\end{example}

The polytopes in the parliament also have an attractive reinterpretation as
toric line bundles.

\begin{remark}
  \label{rem:gens}
  For each flat $F$ in the matroid $\Mat(\cE)$, the associated $T$-invariant
  divisor on $X$ is defined to be
  $D_F := a_1(F) \, D_1 + a_2(F) \, D_2 + \dotsb + a_n(F) \, D_n$, where
  $a_i(F) := \max\{ j \in \ZZ : \text{$\Span(F) \subseteq E^i(j)$} \}$.  In
  particular, each flat $F$ gives rise to an toric line bundle $\cO_{X}(D_F)$.
  When a flat is defined by a single element $\ee$ in the ground set of
  $\Mat(\cE)$, the polytope corresponding to $\cO_{X}(D_{\ee})$ is simply the
  polytope $P_\ee$ from the parliament for $\cE$.  By construction, there is a
  natural map from the filtrations of the toric vector bundle
  $\bigoplus_\ee \cO_{X}(D_{ \ee })$ onto the filtrations for the toric vector
  bundle $\cE$.  Hence, the equivalence of categories yields a canonical
  surjective homomorphism
  \[
    \eta \colon \bigoplus_{\ee} \cO_X(D_{\ee}) \to \cE \, ,
  \] 
  where the sum is over all elements $\ee$ in the ground set of the matroid
  $\Mat(\cE)$.  Rephrasing Proposition~\ref{pro:sections}, we see that the map
  $\eta$ induces a surjection on global sections.
\end{remark}

Our second example shows that the ground set of the matroid $\Mat(\cE)$ may be
strictly larger than the union $\bigcup_{\sigma \in \Sigma(d)} \sB_\sigma$ of
the bases for $E$ that split the filtrations over the maximal cones.

\begin{example}
  \label{exa:Ref} 
  To describe a toric vector bundle $\cE$ of rank $3$ on $\PP^1 \times \PP^1$,
  we first specify the fan: the unique minimal lattice points generating the
  rays are $\vv_1 = (1,0)$, $\vv_2 = (0,1)$, $\vv_3 = (-1,0)$,
  $\vv_4 = (0,-1)$ and the maximal cones are
  $\sigma_{1,2} = \pos(\vv_1, \vv_2)$, $\sigma_{2,3} = \pos(\vv_2, \vv_3)$,
  $\sigma_{3,4} = \pos(\vv_3, \vv_4)$, $\sigma_{1,4} = \pos(\vv_1, \vv_4)$.
  If $\ee_1, \ee_2, \ee_3$ denotes the standard basis of $E = \CC^3$, then the
  decreasing filtrations defining $\cE$ are
  \begin{xalignat*}{2}
    E^{1}(j) & = 
    \left\{
      \renewcommand{\arraystretch}{0.9}
      \renewcommand{\arraycolsep}{1pt}
      \begin{array}{p{70pt}crcl}
        $E$ & \;\; \text{if} \;\; & & j & \leq -1\\
        $\Span ( \ee_1, \ee_2 )$ & \text{if} & -1 < & j & \leq 0 \\
        $\Span ( \ee_1 +\ee_2 )$ & \text{if} & 0 < & j & \leq 1 \\
        $0$ & \text{if} & 1 < & j & \, ,
      \end{array}
    \right. &
     E^{3}(j) & = 
    \left\{
      \renewcommand{\arraystretch}{0.9}
      \renewcommand{\arraycolsep}{1pt}
      \begin{array}{p{70pt}crcl}
        $E$ & \;\; \text{if} \;\; & & j & \leq -1 \\
        $\Span ( \ee_1, \ee_3 )$ & \text{if} & -1 < & j & \leq 0 \\
        $\Span ( \ee_1+ \ee_3 )$ & \text{if} & 0 < & j & \leq 1 \\
        $0$ & \text{if} & 1 < & j & \, ,
      \end{array}    \right. \\
     E^{2}(j) & = 
    \left\{
      \renewcommand{\arraystretch}{0.9}
      \renewcommand{\arraycolsep}{1pt}
      \begin{array}{p{70pt}crcl}
        $E$ & \;\; \text{if} \;\; & & j & \leq 0 \\
        $\Span ( \ee_2, \ee_3 )$ & \text{if} & 0 < & j & \leq 1 \\
        $\Span ( \ee_2 )$ & \text{if} & \phantom{-}1 < & j & \leq 2 \\
        $0$ & \text{if} & 2 < & j  & \, ,
      \end{array}      \right. &
       E^{4}(j) & = 
    \left\{
      \renewcommand{\arraystretch}{0.9}
      \renewcommand{\arraycolsep}{1pt}
      \begin{array}{p{70pt}crcl}
        $E$ & \;\; \text{if} \;\; & & j & \leq 0 \\
        $\Span ( \ee_2, \ee_3 )$ & \text{if} & 0 < & j & \leq 1 \\
        $\Span ( \ee_2 )$ & \text{if} & \phantom{-}1 < & j & \leq 2 \\
        $0$ & \text{if} & 2 < & j & \, .
      \end{array} \right.
  \end{xalignat*}
  It follows that the ground set of matroid $\Mat(\cE)$ is
  $\{\ee_1, \ee_1+ \ee_2, \ee_1 + \ee_3, \ee_2, \ee_3 \}$.
  Figure~\ref{fig:latticeforE} represents its lattice of flats; the flats
  \begin{figure}[t]
    \begin{tikzpicture}[x=1.35cm, y=0.5cm, line width=1.25pt]
      \draw (0,0) -- (-4,1.5);
      \draw (0,0) -- (-2,1.5);
      \draw (0,0) -- (0,1.5);
      \draw (0,0) -- (2,1.5);
      \draw (0,0) -- (4,1.5);
      \draw (0,6.5) -- (-3,4.5);
      \draw (0,6.5) -- (-1,4.5);
      \draw (0,6.5) -- (1,4.5);
      \draw (0,6.5) -- (3,4.5);
      \draw (0,6.5) -- (5,4.5);
      \draw (-4,1.5) -- (-5,4.5);
      \draw (-4,1.5) -- (-3,4.5);
      \draw (-2,1.5) -- (-5,4.5);
      \draw (-2,1.5) -- (-1,4.5);
      \draw (-2,1.5) -- (1,4.5);
      \draw (0,1.5) -- (-3,4.5);
      \draw (0,1.5) -- (-1,4.5);
      \draw (0,1.5) -- (3,4.5);
      \draw (2,1.5) -- (-5,4.5);
      \draw (2,1.5) -- (3,4.5);
      \draw (2,1.5) -- (5,4.5);
      \draw (4,1.5) -- (-3,4.5);
      \draw (4,1.5) -- (1,4.5);
      \draw (4,1.5) -- (5,4.5);
      \draw (0,6.5) -- (-5,4.5);
      \node[fill=white, ellipse, inner sep=0.5pt] () at (0,0) {\tiny $0$};
      \node[fill=white, ellipse, inner sep=0.5pt] () at (-4,1.5) {\tiny
        $\Span(\ee_1)$};
      \node[fill=white, ellipse, inner sep=0.5pt] () at (-2,1.5) {\tiny
        $\Span(\ee_1+\ee_2)$};
      \node[fill=white, ellipse, inner sep=0.5pt] () at (0,1.5) {\tiny
        $\Span(\ee_1+\ee_3)$};
      \node[fill=white, ellipse, inner sep=0.5pt] () at (2,1.5) {\tiny
        $\Span(\ee_2)$};
      \node[fill=white, ellipse, inner sep=0.5pt] () at (4,1.5) {\tiny
        $\Span(\ee_3)$};
      \node[fill=white, ellipse, inner sep=0.5pt] () at (-5,4.5) {\tiny
        $\Span(\ee_1, \ee_1 + \ee_2, \ee_2)$};
      \node[fill=white, rectangle, inner sep=0.5pt] () at (-1,4.5) {\tiny
        \textcolor{blue}{$\Span(\ee_1 + \ee_2, \ee_1 + \ee_3)$}};
      \node[fill=white, ellipse, inner sep=0.5pt] () at (1,4.5) {\tiny
        \textcolor{blue}{$\Span(\ee_1+\ee_2, \ee_3)$}};
      \node[fill=white, ellipse, inner sep=0.5pt] () at (3,4.5) {\tiny
        \textcolor{blue}{$\Span(\ee_1 +\ee_3, \ee_2)$}};
      \node[fill=white, ellipse, inner sep=0.5pt] () at (5,4.5) {\tiny
        $\Span(\ee_2, \ee_3)$};
     \node[fill=white, rectangle, inner sep=0.8pt] () at (-3,4.5) {\tiny
        $\Span(\ee_1, \ee_1 + \ee_3, \ee_3)$};
      \node[fill=white, ellipse, inner sep=0.5pt] () at (0,6.5) {\tiny $E
        \cong \Span(\ee_1, \ee_1+ \ee_2, \ee_1 + \ee_3, \ee_2, \ee_3)$};
    \end{tikzpicture}
    \caption{Hasse diagram for the lattice of flats}
    \label{fig:latticeforE}      
  \end{figure}
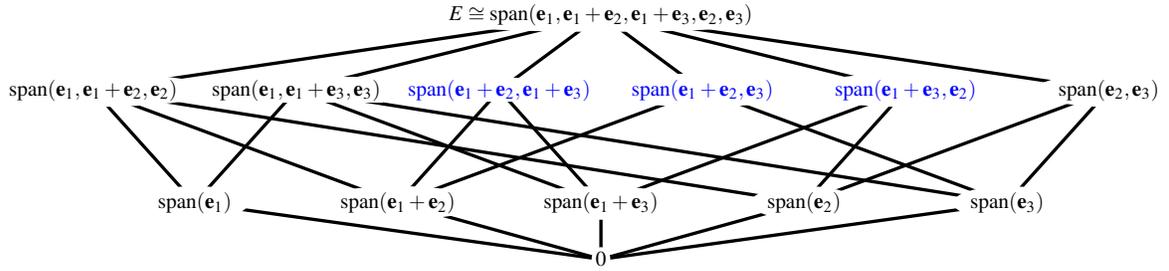
  appearing in $\Mat(\cE)$ but not in $\Poset(\cE)$ are shaded.  
  On each maximal cone, the associated characters and the unique choice of
  compatible basis are
  \begin{xalignat*}{2}
    \bm{u}(\sigma_{1,2}) &= \{ (1, 0), (0,2), (-1, 1) \} \, , &
    \sB_{\sigma_{1,2}} &= \{ \ee_1 + \ee_2, \ee_2, \ee_3 \} \, ,\\
    \bm{u}(\sigma_{2,3}) &= \{ (-1, 0), (1,2), (0, 1) \} \, , &
    \sB_{\sigma_{2,3}} &= \{ \ee_1 + \ee_3, \ee_2, \ee_3 \} \, ,\\
    \bm{u}(\sigma_{3,4}) &= \{ (-1, 0), (1,-2), (0, -1) \} \, ,&
    \sB_{\sigma_{3,4}} &= \{ \ee_1 + \ee_3, \ee_2, \ee_3 \} \, ,\\
    \bm{u}(\sigma_{1,4}) &= \{ (0, -2), (1,0), (-1, -1) \} \, ,&
    \sB_{\sigma_{1,4}} &= \{ \ee_2, \ee_1 + \ee_2, \ee_3 \} \, .
  \end{xalignat*}
  Hence, the parliament for $\cE$ consists of the following five convex
  polytopes: $P_{\ee_1} = \conv \bigl( (0,0) \bigr)$,
  $P_{\ee_1 + \ee_2} = \conv \bigl( (1,0) \bigr)$,
  $P_{\ee_1 + \ee_3} = \conv \bigl( (-1,0) \bigr)$, $P_{\ee_2} = \varnothing$,
  and $P_{\ee_3} = \varnothing$. Although
  $\ee_1 \not\in \bigcup_{\sigma \in \Sigma(2)} \sB_{\sigma}$, we have
  $\Span(\ee_1) = E^{1}(0) \cap E^{3}(0)$. \hfill $\Diamond$
\end{example}

The lattice points in the parliament of polytopes for a toric vector bundle
correspond to a basis if and only if, for all $\uu \in M$, the subset
$\{ \ee \in E : \uu \in P_\ee \}$ is linearly independent.  The next example
illustrates how a single lattice point can correspond to a dependent
collection of global sections.

\begin{example}
  \label{exa:tangentBundle}
  Consider the tangent bundle $\cT_{\PP^d}$ on $\PP^d$.  The minimal lattice
  points $\vv_i$ generating the $i$\nobreakdash-th ray in the fan of $\PP^d$
  equals the $i$\nobreakdash-th standard basis vector in $\CC^d$ for
  $1 \leq i \leq d$, and the additional ray is generated by
  $\vv_{d+1} := - \vv_1 - \vv_2 - \dotsb - \vv_{d}$.  The maximal cones are
  $\sigma_i := \pos(\vv_1, \vv_2, \dotsc ,\vv_{i-1}, \vv_{i+1}, \vv_{i+2},
  \dotsc, \vv_{d+1})$ for $1 \leq i \leq d+1$; compare with Example~3.1.10 in
  \cite{CLS} or Section~1.4 in \cite{Fulton}.  Following Subsection~2.3.5 of
  \cite{Kly}, we identify the fibre $E$ of $\cT_{\PP^d}$ over the identity of
  the torus $T$ with $N \otimes_\ZZ \CC \cong \CC^d$.  Hence, the vectors
  $\vv_1, \vv_2, \dotsc, \vv_d$ also form the standard basis for $E = \CC^d$
  and the decreasing filtrations defining $\cT_{\PP^d}$ are
  \[
  E^{i}(j) =
    \left\{
      \renewcommand{\arraystretch}{0.9}
      \renewcommand{\arraycolsep}{1pt}
      \begin{array}{lcc}
        E & \;\; \text{if} \;\; &  j \leq 0 \\
        \Span(\vv_i) & \text{if} & j = 1 \\
        0 & \text{if} & j > 1 
      \end{array}
    \right.
  \qquad \text{for $1 \leq i \leq d+1$.}
  \]
  Writing $\ww_1, \ww_2, \dotsc, \ww_d$ for the dual basis of $M$
  corresponding to the basis $\vv_1, \vv_2, \dotsc, \vv_d \in N$, we have
  $\bm{u}(\sigma_i) = \{ \ww_1 \!-\! \ww_i, \ww_2 \!-\! \ww_i, \dotsc,
  \ww_{i-1} \!-\! \ww_i, -\ww_i, \ww_{i+1} \!-\! \ww_i, \ww_{i+2} \!-\! \ww_i,
  \dotsc, \ww_d \!-\! \ww_i \}$ for $1 \leq i \leq d$, and
  $\bm{u}(\sigma_{d+1}) = \{ \ww_1, \ww_2, \dotsc, \ww_d \}$.  Hence, the
  ground set of the matroid $\Mat(\cT_{\PP^d})$ is
  $\{ \vv_1, \vv_2, \dotsc, \vv_{d+1} \}$ and the convex polytopes in the
  parliament for $\cT_{\PP^d}$ are
  \[
    P_{\vv_i} = \bigl\{ \uu \in M \otimes_\ZZ \RR :
    \text{$\langle \uu, \vv_i \rangle \leq 1$ and
      $\langle \uu, \vv_j \rangle \leq 0$ for all $j \neq i$} \bigr\} \, .
  \]  
  The lattice points in the parliament of polytopes for $\cT_{\PP^d}$
  correspond to the following $(d+1)^2$ global sections:
  $\vv_i \otimes \chi^{\ww_j-\ww_i}$ for $1 \leq i, j \leq d$,
  $\vv_i \otimes \chi^{-\ww_i}$ for $1 \leq i \leq d$,
  $\vv_{d+1} \otimes \chi^{\ww_i}$ for $1 \leq i \leq d$, and
  $\vv_{d+1} \otimes \chi^{\mathbf{0}}$.  The origin $\mathbf{0} \in M$ is
  contained in all $d+1$ polytopes, which yields $d+1$ global sections in a
  $d$-dimensional vector space.  Following Remark~\ref{rem:gens}, the flat
  $\{ \vv_i \}$ in the matroid $\Mat(\cT_{\PP^d})$ corresponds to the toric
  line bundle $\cO_{\PP^d}(D_i)$ for $1 \leq i \leq d+1$, and the flat given
  by the unique circuit $\{ \vv_1, \vv_2, \dotsc, \vv_{d+1} \}$ in
  $\Mat(\cT_{\PP^d})$ corresponds to $\cO_{\PP^d}$.  Hence, we obtain the
  short exact sequence
  \[
    0 \longrightarrow \cO_{\PP^d} \longrightarrow \bigoplus_{i=1}^{d+1}
    \cO_{\PP^d}(D_i) \xrightarrow{\;\; \eta \;\;} \cT_{\PP^d} \longrightarrow
    0 \, ,
  \]
  which is dual to the classic Euler sequence; see Theorem~8.1.6 in
  \cite{CLS}.

  When $d = 2$, it is possible to visualize the parliament of polytopes.  In
  this case, the associated characters are
  $\bm{u}(\sigma_1) = \{ (-1,0), (-1,1) \}$,
  $\bm{u}(\sigma_2) = \{ (1,-1), (0,-1) \}$,
  $\bm{u}(\sigma_3) = \{ (1,0), (0,1) \}$, and the convex polytopes are
  $P_{\vv_1} = \conv \bigl( (0,0), (1,0), (1,-1) \bigr)$,
  $P_{\vv_2} = \conv \bigl( (0,0), (0,1), (-1,1) \bigr)$,
  $P_{\vv_3} = \conv \bigl( (0,0), (-1,0), (0,-1) \bigr)$.  In
  Figure~\ref{fig:tangentBundle}, the associated characters are represented by
  \begin{figure}[t]
    \begin{tikzpicture}[x=0.75cm, y=0.75cm, line width=1.25pt]
      \foreach \x in {-2,-1,1,2} {
        \draw[color=white!40!black] (\x, 2pt) -- (\x, -2pt);}
      \foreach \y in {-1,1} {
        \draw[color=white!40!black] (2pt, \y) -- (-2pt, \y);}
      \draw[color=white!40!black, <->] (-2.4, 0) -- (2.4,0) {};
      \draw[color=white!40!black, <->] (0, -1.6) -- (0, 1.8) {};
      \draw[color=blue, fill=blue, opacity=0.5] (0,0) -- (-1,0) -- (0,-1) --
      (0,0) -- cycle;
      \draw[color=blue, fill=blue, opacity=0.5] (0,0) -- (1,0) -- (1,-1) --
      (0,0) -- cycle;
      \draw[color=blue, fill=blue, opacity=0.5] (0,0) -- (0,1) -- (-1,1) --
      (0,0) -- cycle;
      \draw[color=blue] (1,0) -- (-1,0) -- (0,-1) -- (0,1) -- (-1,1) -- (1,-1)
      -- (1,0) -- cycle;
      \node[circle, fill=black, inner sep=2.0pt] () at (0,0) {};
      \node[rectangle, fill=black, inner sep=2.5pt] () at (1,0) {};
      \node[rectangle, fill=black, inner sep=2.5pt] () at (0,1) {};
      \node[star, star points=5, star point ratio = 0.3, fill=black, 
      inner sep=3.5pt] () at (-1,1) {};
      \node[star, star points=5, star point ratio = 0.3, fill=black, 
      inner sep=3.5pt] () at (-1,0) {};
      \node [diamond, fill=black, inner sep=2.0pt] () at (0,-1) {};
      \node [diamond, fill=black, inner sep=2.0pt] () at (1,-1) {};
      \node[color=black] () at (-0.9, -0.7) {$P_{\vv_3}$};
      \node[color=black] () at (-0.5, 1.5) {$P_{\vv_2}$};
      \node[color=black] () at (1.5, -0.6) {$P_{\vv_1}$};
    \end{tikzpicture}
    \caption{The parliament of polytopes for $\cT_{\PP^2}$}
    \label{fig:tangentBundle}
  \end{figure}
  asterisks, diamonds, and squares respectively.  The polytopes are
  represented by shaded triangles and the other lattice point lying in the
  polytopes is represented by a circle.  \hfill $\Diamond$
\end{example}

\section{Globally Generated Toric Vector Bundles}
\label{sec:basepointfree}

\noindent
In this section, we establish our criterion for deciding whether a toric
vector bundle is globally generated. To detect the global generation of a
toric vector bundle $\cE$ from its parliament of polytopes, we need a local
description for a global section in coordinates near a $T$-fixed point.

To achieve this, consider a maximal cone $\sigma \in \Sigma(d)$ and the
$T$-fixed point $x_\sigma$.  By reordering the rays (if necessary), we assume
that $\sigma = \pos(\vv_1, \vv_2, \dotsc, \vv_d)$.  Since $X$ is a smooth
toric variety, the unique minimal generators $\ww_1, \ww_2, \dotsc, \ww_d$ of
the dual cone $\sigma^\vee$ form a $\ZZ$-basis for $M$.  By indexing the
underlying set $\underline{\bm{u}}(\sigma)$ of associated characters, we have
$\underline{\bm{u}}(\sigma) = \{ \uu_{\sigma,1}, \uu_{\sigma,2}, \dotsc,
\uu_{\sigma,s} \} \subset M$, for some integer $s$ satisfying
$1 \leq s \leq r$.  Following Section~6.3 in \cite{Kly}, we identify the fibre
of $\cE$ over the $T$-fixed point $x_\sigma$ with the $\CC$-vector space
\[
  \bigoplus_{\uu \in \underline{\bm{u}}(\sigma)}
  \frac{E^{\sigma}_\uu}{E^{\sigma}_{> \uu}} = \bigoplus_{k = 1}^s
  \frac{E^{\sigma}_{\uu_{\sigma,k}}}{E^{\sigma}_{> \uu_{\sigma,k}}} \cong
  \cE_{x_{\sigma}} \cong \CC^r \, ,
\]
where
$E^{\sigma}_\uu := \bigcap_{i = 1}^{d} E^{i}(\langle \uu, \vv_i \rangle)$ and
$E^{\sigma}_{>\uu} := \sum_{\mathbf{0} \neq \uu' - \uu \in \sigma^\vee}
E^{\sigma}_{\uu'} = \sum_{i = 1}^{d} E^{\sigma}_{\uu+\ww_i}$; see
Section~\ref{sec:bundles}.  The linear subspaces $E^{\sigma}_{\uu}$ and
$E^{\sigma}_{> \uu}$ correspond to flats in the matroid $\Mat(\cE)$ that are
generated by subsets of any compatible basis $\sB_\sigma$ in $\Mat(\cE)$.  For
a given compatible basis $\sB_{\sigma}$, this decomposition of the fibre
yields a partition of the set $\sB_{\sigma}$.  Specifically, we have the
disjoint union
\[
  \sB_{\sigma} = \sB_{\sigma,1} \sqcup \sB_{\sigma,2} \sqcup \dotsb \sqcup
  \sB_{\sigma,s}\,
\] 
where the subset $\sB_{\sigma,k}$ consists of all $\ee \in \sB_{\sigma}$ such
that
$\ee \in E^{\sigma}_{\uu_{\sigma,k}} \setminus E^{\sigma}_{> \uu_{\sigma,k}}$
for $1 \leq k \leq s$. By construction, we see that, for all
$1 \leq k \leq s$, the quotient space
$E^{\sigma}_{\uu_{\sigma,k}} / E^{\sigma}_{> \uu_{\sigma, k}}$ is identified
with the linear subspace $\Span(\sB_{\sigma,k}) \subseteq E$ and the
multiplicity of $\uu_{\sigma,k}$ in the multiset $\bm{u}(\sigma)$ of
associated characters equals the number of elements in $\sB_{\sigma,k}$.  With
these preliminaries, we have the following technical lemma.

\begin{lemma}
  \label{lem:function}
  Let $\sigma = \pos(\vv_1, \vv_2, \dotsc, \vv_d)$ be a maximal cone, let
  $\underline{\bm{u}}(\sigma) = \{ \uu_{\sigma,1}, \uu_{\sigma,2}, \dotsc,
  \uu_{\sigma,s} \}$ be the underlying set of associated characters, and let
  $\sB_{\sigma} = \sB_{\sigma,1} \sqcup \sB_{\sigma,2} \sqcup \dotsb \sqcup
  \sB_{\sigma,s}$ be the corresponding partition of a compatible basis in
  $\Mat(\cE)$.  For each $\ee \in \sB_{\sigma}$, consider the continuous
  piecewise linear function on the fan $\Sigma$ defined by
  $\varphi_{\ee}(\vv_i) := \max\{ j \in \ZZ : \ee \in E^{i}(j) \}$ for all
  $1 \leq i \leq n$.  If $\ee \in \sB_{\sigma, k}$ for some $1 \leq k \leq s$,
  then we have $\varphi_{\ee}(\vv_i) = \langle \uu_{\sigma, k}, \vv_i \rangle$
  for all $1 \leq i \leq d$.  In particular, if $\ee \in \sB_{\sigma,k}$ and
  $\uu_{\sigma, k} \in P_{\ee}$, then the lattice point $\uu_{\sigma, k}$ is a
  vertex of the polytope $P_{\ee}$.
\end{lemma}

\begin{proof} 
  Fix an index $k$ such that $1 \leq k \leq s$ and an element
  $\ee \in \sB_{\sigma,k}$.  Since
  $\ee \in \sB_{\sigma,k} \subset E^{\sigma}_{\uu_{\sigma,k}}$, it follows
  that $\ee \in E^{i}(\langle \uu_{\sigma,k}, \vv_i \rangle)$ for all
  $1 \leq i \leq d$, so
  $\max\{ j \in \ZZ : \ee \in E^{i}(j) \} \geq \langle \uu_{\sigma,k}, \vv_i
  \rangle$ for all $1 \leq i \leq d$.  Suppose that, for some index $i$
  satisfying $1 \leq i \leq d$, we have
  $\max\{ j \in \ZZ : \ee \in E^{i}(j) \} > \langle \uu_{\sigma, k}, \vv_i
  \rangle$.  It would follow that
  $\ee \subset E^{\sigma}_{\uu_{\sigma,k} + \ww_\ell}$ for some minimal
  generator $\ww_\ell$ of the dual cone $\sigma^\vee$.  However, this would
  imply that $\ee \in E^{\sigma}_{>\uu_{\sigma,k}}$ which contradicts the
  definition of $\sB_{\sigma,k}$.  Therefore, we conclude that
  $\max\{ j \in \ZZ : \ee \in E^{i}(j) \} = \langle \uu_{\sigma,k}, \vv_i
  \rangle$ for all $1 \leq i \leq d$.  Moreover, when
  $\ee \in \sB_{\sigma,k}$, the piecewise linear function
  $\varphi_{\sigma, k}$ is simply the support function for the polytope
  $P_{\ee}$ in the parliament for $\cE$. Thus, if $\ee \in \sB_{\sigma,k}$ and
  $\uu_{\sigma,k} \in P_{\ee}$, then we see that the lattice point
  $\uu_{\sigma, k}$ is a vertex of this polytope.
\end{proof}

We can now give a local description for a global section around the $T$-fixed
point $x_\sigma$.  The affine semigroup ring $\CC[\sigma^\vee \cap M]$ is the
coordinate ring for the affine open set $U_\sigma \subset X$ and is isomorphic
to the polynomial ring $\CC[y_1,y_2, \dotsc, y_d]$ where
$y_i := \chi^{-\ww_i}$ for $1 \leq i \leq d$.  For any compatible basis
$\sB_{\sigma}$ and any vector $\ee' \in E$, there exists unique scalars
$\lambda_\ee\in \CC$, for all $\ee \in \sB_{\sigma}$, such that
$\ee' = \sum_{\ee \in \sB_{\sigma}} \lambda_{\ee} \, \ee$.  By
Proposition~\ref{pro:sections}, a $T$-equivariant global section of $\cE$ has
the form $\ee' \otimes \chi^{-\uu}$, where $\ee' \in E$ and $\uu \in M$.
Hence, the section $\ee' \otimes \chi^{-\uu}$ is given in local coordinates
near $x_\sigma$ by
\begin{equation}
  \label{eq:local}
  \tag{$\ast$}
  \sum_{\ee \in \sB_{\sigma}} \lambda_\ee \biggl( \ee \otimes \prod_{i
    = 1}^{d} y_i^{- \langle \uu, \vv_i \rangle + \varphi_{\ee}(\vv_i)} \biggr) \, .
  \setcounter{equation}{3}
\end{equation}
Using this local description, we characterize the global generation of a toric
vector bundle via its parliament of polytopes.

\begin{theorem2}
  A toric vector bundle $\cE$ is globally generated if and only if, for all
  $\sigma \in \Sigma(d)$, the associated character $\bm{u}(\sigma)$ are
  vertices of polytopes in the parliament and the elements indexing these
  polytopes form a basis in the matroid $\Mat(\cE)$.
\end{theorem2}

\begin{proof}[Proof of Theorem~\ref{thm:main1}] 
  As Proposition~\ref{pro:sections} shows, the toric vector bundle $\cE$ has a
  $T$-equivariant basis of global sections.  Hence, the locus in the toric
  variety $X$ on which all global sections vanish is closed and $T$-invariant.
  Since $X$ is complete, it follows that the toric vector bundle $\cE$ is
  globally generated if and only if it is globally generated at every
  $T$-fixed point.

  Fix a maximal cone
  $\sigma = \pos(\vv_1, \vv_2, \dotsc, \vv_d) \in \Sigma(d)$, let
  $\underline{\bm{u}}(\sigma) = \{ \uu_{\sigma,1}, \uu_{\sigma,2}, \dotsc,
  \uu_{\sigma,s} \}$ be the underlying set of associated characters, and let
  $\sB_{\sigma} = \sB_{\sigma,1} \sqcup \sB_{\sigma,2} \sqcup \dotsb \sqcup
  \sB_{\sigma,s}$ be the corresponding partition of a compatible basis in the
  matroid $\Mat(\cE)$.  The toric vector bundle $\cE$ is globally generated at
  the $T$-fixed point $x_{\sigma}$ if and only if the evaluation map
  \[
    \operatorname{ev}_{\sigma} \colon H^0(X,\cE) \to H^0(X, \cE \otimes
    \cO_X/{\mathfrak{m}}_{x_{\sigma}}) \cong \Span(\sB_{\sigma})
  \] 
  is surjective.  Since a $T$-equivariant global section
  $\ee' \otimes \chi^{-\uu}$ is given in local coordinates near $x_\sigma$ by
  \eqref{eq:local}, its evaluation at the $T$-fixed point $x_{\sigma}$ is
  given by
  \[
    \sum_{\ee \in \sB_\sigma} \lambda_\ee \biggl( \ee \otimes \prod_{i =
      1}^{d} y_i^{- \langle \uu, \vv_i \rangle + \varphi_{\ee}(\vv_i)} \biggr)
    \bigg|_{y_1 = y_2 = \dotsb = y_d = 0} \, .
  \]
  The $\ee$-th summand in this expression has neither a zero nor a pole at
  $(y_1, y_2, \dotsc, y_d) = (0,0, \dotsc, 0)$ if and only if we have
  $- \langle \uu, \vv_i \rangle + \varphi_{\ee}(\vv_i) = 0$ for all
  $1 \leq i \leq d$.  By Lemma~\ref{lem:function}, it follows that there
  exists an index $k$ such that $\ee \in \sB_{\sigma, k}$ and
  $\uu = \uu_{\sigma,k}$.  In this case, the lattice point $\uu_{\sigma, k}$
  is also a vertex of the polytope $P_{\ee}$ in the parliament for $\cE$.
  Hence, the image of a $T$-equivariant global section under the evaluation
  map $\operatorname{ev}_{\sigma}$ is nonzero in the fibre at $x_\sigma$ if
  and only if the global section has the form
  $\sum_{\ee \in \sB_{\sigma, k}} \lambda_{\ee} (\ee \otimes
  \chi^{-\uu_{\sigma, k}})$, for some $1 \leq k \leq s$, which evaluates to
  $\sum_{\ee \in \sB_{\sigma, k}} \lambda_\ee \, \ee \in \cE_{x_\sigma}$.
  Therefore, the evaluation map $\operatorname{ev}_{\sigma}$ is surjective if
  and only if there exists a compatible basis $\sB_{\sigma}$ such that each
  $\ee \otimes \chi^{- \uu_{\sigma,k}}$, for $\ee \in \sB_{\sigma, k}$ and
  $1 \leq k \leq s$, is a global section.
\end{proof}

Using Theorem~\ref{thm:main1}, we create a low-rank toric vector bundle on
$\PP^2$ that is not globally generated; Example~\ref{exa:notBPF2} will show
that this low-rank toric vector bundle is also ample.

\begin{example} 
  \label{exa:notBPF} 
  To describe a second toric vector bundle $\cF$ of rank $3$ on $\PP^2$, we
  use the notation introduced in Example~\ref{exa:tangentBundle}.
  Specifically, the minimal lattice points generating the rays in the fan are
  $\vv_1 = (1,0)$, $\vv_2 = (0,1)$, $\vv_3 = (-1,-1)$, and the maximal cones
  are $\sigma_1 = \pos(\vv_2, \vv_3)$, $\sigma_2 = \pos(\vv_1, \vv_3)$,
  $\sigma_3 = \pos(\vv_1, \vv_2)$.  If $\ee_1, \ee_2, \ee_3$ denotes the
  standard basis of $E = \CC^3$, then the decreasing filtrations defining
  $\cF$ are
  \begin{xalignat*}{2}
    E^{1}(j) & = 
    \left\{
      \renewcommand{\arraystretch}{0.9}
      \renewcommand{\arraycolsep}{1pt}
      \begin{array}{lcrcl}
        E & \;\; \text{if} \;\; & & j & \leq -1 \\
        \Span ( \ee_1, \ee_2 ) & \text{if} & -1 < & j & \leq 0 \\
        \Span ( \ee_1 ) & \text{if} & 0 < & j & \leq 4 \\
        0 & \text{if} & 4 < & j & \, ,
      \end{array}
    \right. &
    E^{3}(j) &= 
    \left\{
      \renewcommand{\arraystretch}{0.9}
      \renewcommand{\arraycolsep}{1pt}
      \begin{array}{lcrcl}
        E & \;\; \text{if} \;\; & & j & \leq -1 \\
        \Span ( \ee_2 - \ee_3, \ee_1 - \ee_2 ) & \text{if} & -1 < & j &
        \leq 2 \\
        \Span ( \ee_1 -\ee_2 ) & \text{if} & 2 < & j & \leq 3 \\
        0 & \text{if} & 3 < & j & \, ,
      \end{array}
    \right. \\
    E^{2}(j) &= 
    \left\{
      \renewcommand{\arraystretch}{0.9}
      \renewcommand{\arraycolsep}{1pt}
      \begin{array}{lcrcl}
        E & \;\; \text{if} \;\; & & j & \leq -2 \\
        \Span ( \ee_2, \ee_3 ) & \text{if} & -2 < & j & \leq 0 \\
        \Span ( \ee_3 ) & \text{if} & 0 < & j & \leq 3 \\
        0 & \text{if} & 3 < & j & \, .
      \end{array}
      \right. 
  \end{xalignat*}
  It follows that the ground set of the matroid $\Mat(\cF)$ is
  $\{ \ee_1, \ee_1 - \ee_2, \ee_2, \ee_2 - \ee_3, \ee_3 \}$. On each maximal
  cone, the associated characters and the unique choice of compatible bases
  are
  \begin{xalignat*}{2}
    \bm{u}(\sigma_1) &= \{ (-1, -2), (-2,0), (-2, 3) \} \, , & 
    \sB_{\sigma_1} &= \{ \ee_1 - \ee_2, \ee_2 - \ee_3, \ee_3 \} \, , \\
    \bm{u}(\sigma_2) &= \{ (4, -3), (0,-3), (-1, -1) \} \, , & 
    \sB_{\sigma_2} &= \{ \ee_1, \ee_1 - \ee_2, \ee_2 - \ee_3 \} \, , \\
    \bm{u}(\sigma_3) &= \{ (4, -2), (0,0), (-1, 3) \} \, , & 
    \sB_{\sigma_3} &= \{ \ee_1, \ee_2, \ee_3 \} \, ,
  \end{xalignat*}      
  so the convex polytopes in the parliament for $\cF$ are
  \begin{xalignat*}{2}
    P_{\ee_1} &= \conv \bigl( (3,-2), (4,-2), (4,-3) \bigr) \, , &
    P_{\ee_2 - \ee_3} &= \conv \bigl( (-2,0), (-1,0), (-1,-1) \bigr) \, , \\
    P_{\ee_1 - \ee_2} &= \conv \bigl( (-1,-2), (0,-2), (0,-3) \bigr) \, , &
    P_{\ee_3} &= \conv \bigl( (-2,3), (-1,3), (-1,2) \bigr) \, , \\
    P_{\ee_2} &= \varnothing \, .
  \end{xalignat*}
  In Figure~\ref{fig:notSpanned}, the associated characters are represented by
  \begin{figure}[t]
    \begin{tikzpicture}[x=0.75cm, y=0.75cm, line width=1.25pt]
      \foreach \x in {-5,-4,-3,-2,-1,1,2,3,4,5} {
        \draw[color=white!40!black] (\x, 2pt) -- (\x, -2pt);}
      \foreach \y in {-3,-2,-1,1,2,3} {
        \draw[color=white!40!black] (2pt, \y) -- (-2pt, \y);}
      \draw[color=white!40!black, <->] (-5.4, 0) -- (5.4,0) {};
      \draw[color=white!40!black, <->] (0, -3.4) -- (0, 3.4) {};
      \draw[color=blue, fill=blue, opacity=0.5] (3,-2) -- (4,-2) -- (4,-3) --
      (3,-2) -- cycle;
      \draw[color=blue, fill=blue, opacity=0.5] (-1,-2) -- (0,-2) -- (0,-3) --
      (-1,-2) -- cycle;
      \draw[color=blue, fill=blue, opacity=0.5] (-2,3) -- (-1,3) -- (-1,2) --
      (-2,3) -- cycle;
      \draw[color=blue, fill=blue, opacity=0.5] (-2,0) -- (-1,0) -- (-1,-1) --
      (-2,0) -- cycle;
      \draw[color=blue] (4,-2) -- (-1,-2) -- (0,-3) -- (0,0) -- (-2,0) --
      (-1,-1) -- (-1,3) -- (-2,3) -- (4,-3) -- (4,-2) -- cycle;
      \node[circle, fill=black, inner sep=2.0pt] () at (-1,2) {};
      \node[circle, fill=black, inner sep=2.0pt] () at (-1,0) {};
      \node[circle, fill=black, inner sep=2.0pt] () at (0,-2) {};
      \node[circle, fill=black, inner sep=2.0pt] () at (3,-2) {};
      \node[rectangle, fill=black, inner sep=2.5pt] () at (4,-2) {};
      \node[rectangle, fill=black, inner sep=2.5pt] () at (0,0) {};
      \node[rectangle, fill=black, inner sep=2.5pt] () at (-1,3) {};
      \node[rectangle, fill=white, inner sep=1pt] () at (0,0) {};
      \node[star, star points=5, star point ratio = 0.3, fill=black, 
      inner sep=3.5pt] () at (-1,-2) {};
      \node[star, star points=5, star point ratio = 0.3, fill=black, 
      inner sep=3.5pt] () at (-2,0) {};
      \node[star, star points=5, star point ratio = 0.3, fill=black, 
      inner sep=3.5pt] () at (-2,3) {};
      \node[diamond, fill=black, inner sep=2.0pt] () at (4,-3) {};
      \node[diamond, fill=black, inner sep=2.0pt] () at (0,-3) {};
      \node[diamond, fill=black, inner sep=2.0pt] () at (-1,-1) {};
      \node[regular polygon, regular polygon sides=3, fill=red!60!black, 
      inner sep=1.5pt] () at (1,-1) {};
      \node[color=black] () at (3.1,-2.7) {$P_{\ee_1}$};
      \node[color=black] () at (-1.9,2.3) {$P_{\ee_3}$};
      \node[color=black] () at (-2,-0.7) {$P_{\ee_2 - \ee_3}$};
      \node[color=black] () at (-1,-2.7) {$P_{\ee_1 - \ee_2}$};
    \end{tikzpicture}
    \caption{The parliament of polytopes for $\cF$}
    \label{fig:notSpanned}
  \end{figure}
  asterisks, diamonds, and squares respectively. The polytopes are represented
  by shaded triangles and the other lattice points lying in the polytopes are
  represented by circles.  The square with empty interior represents the
  unique associated character $(0,0)$ that does not lie in any of the
  polytopes.  Therefore, Theorem~\ref{thm:main1} shows that $\cF$ is not
  globally generated.  \hfill $\Diamond$
\end{example}

\begin{remark}
  Our diagrams for parliaments of polytopes, such as the one appearing in
  Figure~\ref{fig:notSpanned}, have at least some superficial similarities to
  the twisted polytopes appearing in Section~6 of \cite{KT}.  It would be
  interesting to develop a more substantive connection.
\end{remark}

If all the polytopes in the parliament for a toric vector bundle $\cE$
correspond to globally-generated line bundles, then the toric vector bundle
$\cE$ itself is globally-generated.  However, the converse is false.
We close this section with a globally-generated toric vector bundle in which
some members of the parliament of polytopes do not correspond to
globally-generated line bundles. 

\begin{example}  
  \label{exa:nonvanishing}
  To describe our toric vector bundle $\cG$ of rank $2$ on the first
  Hirzebruch surface $X = \PP\bigl( \cO_{\PP^1} \oplus \cO_{\PP^1}(1) \bigr)$,
  we first specify the fan.  The minimal lattice points generating the rays
  are $\vv_1 = (1,0)$, $\vv_2 = (0,1)$, $\vv_3 = (-1,1)$, $\vv_4 = (0,-1)$,
  and the maximal cones are $\sigma_{1,2} = \pos(\vv_1, \vv_2)$,
  $\sigma_{2,3} = \pos(\vv_2,\vv_3)$, $\sigma_{3,4} = \pos(\vv_3, \vv_4)$,
  $\sigma_{1,4} = \pos(\vv_1,\vv_4)$.  If $\ee_1, \ee_2$ denotes the standard
  basis of $E = \CC^2$, then the decreasing filtrations defining $\cG$ are
  \begin{xalignat*}{2}
    E^{1}(j) &= 
    \left\{
      \renewcommand{\arraystretch}{0.9}
      \renewcommand{\arraycolsep}{1pt}
      \begin{array}{lcrcl}
        E & \;\; \text{if} \;\; & & j & \leq -2 \\
        \Span ( \ee_1 ) & \text{if} & -2 < & j & \leq 4 \\
        0 & \text{if} &  4 < & j &
      \end{array}
    \right. &
    E^{3}(j) &= 
    \left\{
      \renewcommand{\arraystretch}{0.9}
      \renewcommand{\arraycolsep}{1pt}
      \begin{array}{p{70pt}crcl}
        $E$ & \;\; \text{if} \;\;  & \phantom{-1 <} & j & \leq 0 \\
        $\Span ( \ee_2 )$ & \text{if} & 0 < & j & \leq 5 \\
        $0$ & \text{if} & 5 < & j
      \end{array}
    \right. \\
    E^{2}(j) &= 
    \left\{
      \renewcommand{\arraystretch}{0.9}
      \renewcommand{\arraycolsep}{1pt}
      \begin{array}{lcrcl}
        E & \;\; \text{if} \;\; & & j & \leq 2 \\
        \Span ( \ee_1 ) & \text{if} & \phantom{-}2 < & j & \leq 3 \\
        0 & \text{if} & 3 < & j
      \end{array}
    \right. &
    E^{4}(j) &= 
    \left\{
      \renewcommand{\arraystretch}{0.9}
      \renewcommand{\arraycolsep}{1pt}
      \begin{array}{p{70pt}crcl}
        $E$ & \;\; \text{if} \;\; &  & j & \leq -1 \\
        $\Span ( \ee_1 + \ee_2 )$ & \text{if} & -1 < & j & \leq 3 \\
        $0$ & \text{if} & 3 < & j & \, . 
      \end{array}
    \right.
  \end{xalignat*}
  It follows that the ground set of the matroid $\Mat(\cG)$ is
  $\{ \ee_1, \ee_1 + \ee_2, \ee_2 \}$.  On the maximal cones, the associated
  characters and a choice of compatible bases are
  \begin{xalignat*}{2}
    \bm{u}(\sigma_{1,2}) &= \{(-2, 2), (4,3) \}
    & \sB_{\sigma_{1,2}} &= \{\ee_2, \ee_1 \} \\
    \bm{u}(\sigma_{2,3}) &= \{ (-3, 2), (3,3) \}
    & \sB_{\sigma_{2,3}} &= \{\ee_2, \ee_1 \} \\
    \bm{u}(\sigma_{3,4})& = \{ (-4, 1), (-3,-3) \}
    & \sB_{\sigma_{3,4}}& = \{\ee_2, \ee_1+\ee_2 \} \\
    \bm{u}(\sigma_{1,4}) &= \{(-2, -3), (4,1) \}
    &  \sB_{\sigma_{1,4}} &= \{\ee_1+\ee_2, \ee_1 \} \, .
  \end{xalignat*}
  The convex polytopes in the parliament for $\cG$ are
  \begin{xalignat*}{2}
    P_{\ee_1} &= \conv \bigl( (1,1), (3,3), (4,3), (4,1) \bigr) \, , &
    P_{\ee_2} &= \conv \bigl( (-4,1), (-3,2), (-2,2), (-2,1) \bigr) \, , \\
    P_{\ee_1+\ee_2} &= \conv \bigl( (-3,-3), (-2,-2), (-2,-3) \bigr) \, . 
  \end{xalignat*}
  The set $\{ \ee_1 + \ee_2, \ee_1 \}$ also forms a compatible basis on
  $\sigma_{1,2}$, but the character $(-2,2)$ does not belong to the polytope
  $P_{\ee_1 + \ee_2}$.  In Figure~\ref{fig:nonvanishing}, the associated
  characters are represented by squares, asterisks, diamonds, and pentagons
  \begin{figure}[ht]
    \begin{tikzpicture}[x=0.75cm, y=0.75cm, line width=1.25pt]
      \foreach \x in {-4,-3,-2,-1,1,2,3,4} {
        \draw[color=white!40!black] (\x, 2pt) -- (\x, -2pt);}
      \foreach \y in {-3,-2,-1,1,2,3} {
        \draw[color=white!40!black] (2pt, \y) -- (-2pt, \y);}
      \draw[color=white!40!black, <->] (-4.4, 0) -- (4.4,0) {};
      \draw[color=white!40!black, <->] (0, -3.4) -- (0, 3.4) {};
      \draw[color=blue, fill=blue, opacity=0.5] (1,1) -- (3,3) -- (4,3) --
      (4,1) -- (1,1) -- cycle;
      \draw[color=blue, fill=blue, opacity=0.5] (-4,1) -- (-3,2) -- (-2,2) --
      (-2,1) -- (-4,1) -- cycle;
      \draw[color=blue, fill=blue, opacity=0.5] (-3,-3) -- (-2,-2) -- (-2,-3)
      -- (-3,-3) -- cycle;
      \draw[color=blue] (4,3) -- (3,3) -- (-3,-3) -- (-2,-3) -- (-2,2) --
      (-3,2) -- (-4,1) -- (4,1) -- (4,3) -- cycle;
      \node[circle, fill=black, inner sep=2.0pt] () at (1,1) {};
      \node[circle, fill=black, inner sep=2.0pt] () at (2,1) {};
      \node[circle, fill=black, inner sep=2.0pt] () at (3,1) {};
      \node[circle, fill=black, inner sep=2.0pt] () at (2,2) {};
      \node[circle, fill=black, inner sep=2.0pt] () at (3,2) {};
      \node[circle, fill=black, inner sep=2.0pt] () at (4,2) {};
      \node[circle, fill=black, inner sep=2.0pt] () at (-2,1) {};
      \node[circle, fill=black, inner sep=2.0pt] () at (-3,1) {};
      \node[circle, fill=black, inner sep=2.0pt] () at (-2,-2) {};
      \node[regular polygon, regular polygon sides=3, fill=red!60!black, 
      inner sep=1.5pt] () at (-1,0) {};
      
      \node[rectangle, fill=black, inner sep=2.5pt] () at (-2,2) {};
      \node[rectangle, fill=black, inner sep=2.5pt] () at (4,3) {};
      \node[star, star points=5, star point ratio = 0.3, fill=black, 
      inner sep=3.5pt] () at (-3,2) {};
      \node[star, star points=5, star point ratio = 0.3, fill=black, 
      inner sep=3.5pt] () at (3,3) {};
      \node[diamond, fill=black, inner sep=2.0pt] () at (-4,1) {};
      \node[diamond, fill=black, inner sep=2.0pt] () at (-3,-3) {};
      \node[regular polygon, regular polygon sides=5, fill=black, 
      inner sep=2.3pt] () at (-2,-3) {}; 
      \node[regular polygon, regular polygon sides=5, fill=black, 
      inner sep=2.3pt] () at (4,1) {};
      \node[color=black] () at (4.7,2.0) {$P_{\ee_1}$};
      \node[color=black] () at (-4.6,1.5) {$P_{\ee_2}$};
      \node[color=black] () at (-1.0,-2.5) {$P_{\ee_1 + \ee_2}$};
    \end{tikzpicture}
    \caption{The parliament of polytopes for $\cG$}
    \label{fig:nonvanishing}
  \end{figure}
  respectively.  The polytopes are represented by shaded regions and the other
  lattice points lying in the polytopes are represented by circles. We see
  that each associated character lies in a unique polytope in the parliament.
  Moreover, for each maximal cone, the elements indexing polytopes containing
  the associated characters are equal to our chosen compatible bases, so
  Theorem~\ref{thm:main1} shows that $\cG$ is globally generated.
  Remark~\ref{rem:gens} shows that the elements
  $\{ \ee_1, \ee_1 + \ee_2, \ee_2 \}$ correspond to the toric line bundles
  $\cO_X(4D_1 + 3D_2 - D_4)$, $\cO_X(-2D_1 + 2D_2 + 3D_4)$, and
  $\cO_X(-2D_1 + 2D_2 + 5D_3 - D_4)$ respectively.  The first two line bundles
  are very ample, but the third is not even globally generated.  The third
  line bundle is globally generated at the $T$-fixed points $x_{\sigma_{3,4}}$
  and $x_{\sigma_{1,4}}$, but not at the other $T$-fixed points. \hfill
  $\Diamond$
\end{example}


\section{Contrasting Notions of Positivity}
\label{sec:ample}

\noindent
In this section, we distinguish the ampleness of a toric vector bundle from
other algebraic notions of positivity.  Following Definition~6.1.1 in
\cite{PAG2}, a vector bundle $\cE$ on $X$ is ample or nef if the tautological
line bundle $\cO_{\PP(\cE)}(1)$ on the projectivized bundle $\PP(\cE)$ is
ample or nef, respectively.  Theorem~2.1 in \cite{HMP} states that \emph{a
  toric vector bundle on a complete toric variety is ample if and only if its
  restriction to every torus-invariant curve is ample}.  This provides the key
tool for recognizing ample toric vector bundles.

To be more precise, consider a $T$-invariant curve $C$ in $X$ corresponding to
the cone $\tau \in \Sigma(d-1)$.  Since $X$ is complete, there are two maximal
cones $\sigma$ and $\sigma'$ in $\Sigma(d)$ that contain $\tau$ and
$C \cong \PP^1$.  Given two lattice points $\uu$ and $\uu'$ in $M$ that agree
as linear functionals on $\tau$, the toric line bundle $\cL_{\uu,\uu'}$ on the
union $U_\sigma \cup U_{\sigma'}$ is constructed by gluing
$\cL_{\uu}|_{U_\sigma}$ and $\cL_{\uu'}|_{U_{\sigma'}}$ via the transition
function $\chi^{\uu-\uu'}$, which is regular and invertible on $U_{\tau}$.  If
the lattice vector $\vv_\tau \in \sigma$ is dual to the primitive generator of
$\tau^{\perp}$, then the line bundle $\cL_{\uu,\uu'}|_C$ is isomorphic to
$\cO_{\PP^1}(\langle \uu, \vv_\tau \rangle D_1 - \langle \uu', \vv_{\tau}
\rangle D_2)$ where $D_1$ and $D_2$ are the irreducible $T$-invariant divisors
on $\PP^1$.  Corollary~5.5 and Corollary~5.10 in \cite{HMP} show that the
restriction $\cE|_C$ splits $T$-equivariantly into a sum of line bundles
$\cL_{\uu_1, \uu_1'}|_C \oplus \cL_{\uu_2, \uu_2'}|_C \oplus \dotsb \oplus
\cL_{\uu_r, \uu_r'}|_C$ and the pairs $(\uu_i^{}, \uu_i')$ are unique up to
reordering.  This pairing can be visualized as line segments parallel to
$\tau^{\perp}$ joining the associated characters in $\bm{u}(\sigma)$ and
$\bm{u}(\sigma')$.  Edges in the parliament of polytopes of $\cE$ are
contained in such line segments, but these line segments may connect disjoint
polytopes.  For each individual summand, we have
$\cL_{\uu, \uu'}|_C \cong \cO_{\PP^1}(a)$ where $\uu - \uu'$ is $a$ times the
primitive generator of $\tau^{\perp}$ that is positive on $\sigma$.
Pictorially, the integer $a$ is the normalized lattice distance between the
associated characters in the one-dimensional lattice
$(\tau^{\perp} + \uu) \cap M$.

To demonstrate this apparatus, we reestablish that the tangent bundle on
projective space is ample; compare with Remark~2.4 and Example~5.6 in
\cite{HMP}.

\begin{example}
  \label{exa:ampleTangent}
  Using the notation from Example~\ref{exa:tangentBundle}, the characters
  associated to the tangent bundle $\cT_{\PP^d}$ are
  $\bm{u}(\sigma_i) = \{ \ww_1 \!-\! \ww_i, \ww_2 \!-\! \ww_i, \dotsc,
  \ww_{i-1} \!-\! \ww_i, - \ww_i, \ww_{i+1} \!-\! \ww_i, \ww_{i+2} \!-\!
  \ww_i, \dotsc, \ww_d \!-\! \ww_i \}$ for $1 \leq i \leq d$, and
  $\bm{u}(\sigma_{d+1}) = \{ \ww_1, \ww_2, \dotsc, \ww_d\}$.  On the
  $T$-invariant curve $C_{i,j}$ corresponding to the cone
  $\tau_{i,j} := \sigma_i \cap \sigma_j \in \Sigma(d-1)$ where
  $1 \leq i < j \leq d$, the characters in $\bm{u}(\sigma_i)$ and
  $\bm{u}(\sigma_j)$ are paired as follows: $(-\ww_i, \ww_i - \ww_j)$,
  $(\ww_j-\ww_i, -\ww_j)$, and $(\ww_k - \ww_i, \ww_k - \ww_j)$ for all
  $k \neq i$ or $j$.  Thus, we deduce that
  $\cT_{\PP^d}|_{C_{i,j}} = \cO_{\PP^1}(D_1+D_2) \oplus \bigl(
  \bigoplus_{j=1}^{d-1} \cO_{\PP^1}(D_2) \bigr) \cong \cO_{\PP^1}(2) \oplus
  \bigl( \bigoplus_{j=1}^{d-1} \cO_{\PP^1}(1) \bigr)$.  A similar calculation
  for the curve $C_{i,d+1}$, which corresponds to the cone
  $\tau_{i, d+1} := \sigma_i \cap \sigma_{d+1} \in \Sigma(d-1)$ where
  $1 \leq i \leq d$, yields
  $\cT_{\PP^d}|_{C_{i,d+1}} \cong \cO_{\PP^1}(2) \oplus \bigl(
  \bigoplus_{j=1}^{d-1} \cO_{\PP^1}(1) \bigr)$.  When $d = 2$, we also see
  from Figure~\ref{fig:tangentBundle} that the normalized lattice distance
  between matched pairs of associated characters is either $1$ or $2$.  Since
  the restriction to every $T$-invariant curve is ample, we conclude that
  $\cT_{\PP^d}$ is ample. \hfill $\Diamond$
\end{example}

With these tools, we can also prove directly that the cotangent bundle on a
smooth toric variety is never ample; compare with Section~6.3B in \cite{PAG2}.

\begin{example}
  \label{exa:cotangent}
  Let $\Omega_X$ be the cotangent bundle on a smooth toric variety $X$, and
  let $\Sigma$ be the fan of $X$.  Identifying the fibre $E$ over the identity
  of the torus $T$ with $M \otimes_\ZZ \CC \cong \CC^d$ as done in
  Section~2.3.5 in \cite{Kly}, the decreasing filtrations for $\Omega_X$ are
  \[
  E^{i}(j) = 
  \left\{
      \renewcommand{\arraystretch}{0.9}
      \renewcommand{\arraycolsep}{1pt}
      \begin{array}{lcl}
        E & \;\; \text{if} \;\; & j \leq -1 \\
        \vv_i^{\perp} & \text{if} & j = 0 \\
        0 & \text{if} &  j > 0 
      \end{array}
    \right. 
    \qquad \text{for all $1 \leq i \leq n$.}
  \]
  Consider two adjacent cones $\sigma, \sigma' \in \Sigma(d)$.  Since $X$ is
  smooth, we have $\sigma = \pos(\vv_1,\vv_2, \dotsc, \vv_d)$ where
  $\vv_1, \vv_2, \dotsc, \vv_d$ is a basis for $N$.  We may assume that
  $\sigma' = \pos(\vv_1, \vv_2, \dotsc, \vv_{d-1}, \vv_{d+1})$ where
  $\vv_{d+1} = a_1 \vv_1 + a_2 \vv_2 + \dotsb + a_{d-1} \vv_{d-1}- \vv_d$ for
  some $a_j \in \ZZ$.  If $\ww_1, \ww_2, \dotsc, \ww_d \in M$ form the dual
  basis to $\vv_1, \vv_2, \dotsc, \vv_d$, then the associated characters are
  $\bm{u}(\sigma) = \{ -\ww_1, -\ww_2, \dotsc, - \ww_d \}$ and
  $\bm{u}(\sigma') = \{ -\ww_1 - a_1 \ww_d, - \ww_2 - a_2 \ww_d, \dotsc,
  -\ww_{d-1} - a_{d-1} \ww_{d-1}, \ww_d \}$.  Along the $T$-invariant curve
  $C$ corresponding to the cone $\tau = \sigma \cap \sigma' \in \Sigma(d-1)$,
  the characters are paired as follows: $(-\ww_d, \ww_d)$ and
  $(-\ww_i,-\ww_i - a_i \ww_d)$ for $1 \leq i \leq d-1$.  Therefore, we obtain
  $\Omega_X |_{C} \cong \cO_{\PP^1}(-2) \oplus \bigl( \bigoplus_{i=1}^{d-1}
  \cO_{\PP^1}(a_i) \bigr)$ which implies that $\Omega_X$ is not ample. \hfill
  $\Diamond$
\end{example}

More significantly, we next exhibit an ample toric vector bundle on a smooth
toric variety that is not globally generated.  In particular, this supersedes
Examples~4.15--4.17 in \cite{HMP} and answers the second part of Question~7.5
in \cite{HMP}.

\begin{example}
  \label{exa:notBPF2}
  Consider the toric vector bundle $\cF$ on $\PP^2$ appearing in
  Example~\ref{exa:notBPF}.  Having already established that $\cF$ is not
  globally generated, it remains to show that $\cF$ is ample.  Let $C_k$
  denote the $T$-invariant curve in $\PP^2$ corresponding to the cone
  $\tau_{i,j} := \sigma_i \cap \sigma_j \in \Sigma(d-1)$ where
  $\{i,j,k\} = \{1,2,3\}$.  From the line segments in
  Figure~\ref{fig:notSpanned} joining diamonds to squares, we see that the
  characters in $\bm{u}(\sigma_2)$ and $\bm{u}(\sigma_3)$ are paired on $C_1$
  as $\bigl( (-1,3), (-1,-1) \bigr)$, $\bigl( (0,0), (0,-3) \bigr)$,
  $\bigl( (4,-2), (4,-3) \bigr)$, so we obtain
  \[
  \cF|_{C_1} = \cO_{\PP^1}(3D_1 + D_2) \oplus \cO_{\PP^1}(3D_2) \oplus
  \cO_{\PP^1}(-2D_1 + 3D_2) \cong \cO_{\PP^1}(4) \oplus \cO_{\PP^1}(3) \oplus
  \cO_{\PP^1}(1) \, .
  \]  
  Similar calculations give
  $\cF|_{C_2} \cong \cO_{\PP^1}(5) \oplus \cO_{\PP^1}(2) \oplus
  \cO_{\PP^1}(1)$ and
  $\cF|_{C_3} \cong \cO_{\PP^1}(6) \oplus \cO_{\PP^1}(1) \oplus
  \cO_{\PP^1}(1)$.  Since the restriction to every $T$-invariant curve is
  ample, the toric vector bundle $\cF$ is ample. \hfill $\Diamond$
\end{example}

The vector bundle $\cF$ has minimal rank among all ample toric vector bundles
on $\PP^2$ that are not globally generated.  More than that, the ensuing
proposition proves that, for low-rank toric vector bundles on $\PP^d$, nef is
equivalent to globally generated.

\begin{proposition} 
  \label{pro:lowRank}
  If $\cE$ is a toric vector bundle on $\PP^d$ with rank at most $d$, then
  $\cE$ is globally generated if and only if it is nef.
\end{proposition} 

\begin{proof}  
  As follows from Example~1.4.5 in \cite{PAGI}, every globally generated
  vector bundle is nef, so it suffices to prove the converse implication.
  Moreover, a line bundle on a complete toric variety is nef if and only if it
  is globally generated; see Theorem~6.3.13 in \cite{CLS}.  Hence, the
  proposition follows immediately when $\cE$ splits as a direct sum of line
  bundles.  If the rank of $\cE$ is less than $d$, then Corollary~3.5 in
  \cite{Kan} or Corollary~6.1.5 in \cite{Kly} imply that $\cE$ splits into a
  direct sum of line bundles.  Therefore, we may assume that $\cE$ is
  indecomposable and has rank equal to $d$.

  Under these hypotheses, Theorem~4.6 in \cite{Kan} establishes that $\cE$ is
  isomorphic to either $\cQ(\ell)$ or $\cQ^*(\ell)$ for some $\ell \in \ZZ$,
  where $\cQ$ is defined by the short exact sequence
  \[
  0 \longrightarrow \cO_{\PP^d} \xrightarrow{\;\; \left[ 
      \begin{smallmatrix}
        y_1^{a_1} & y_2^{a_{2}} & \dotsb & y_{d+1}^{a_{d+1}}
      \end{smallmatrix} \right] \;\;} 
  \bigoplus_{k=1}^{d+1} \cO_{\PP^d}(a_k D_k) \longrightarrow \cQ
  \longrightarrow 0 \, ,
  \]
  $a_1, a_2, \dotsc, a_{d+1}$ are positive integers, and
  $D_1, D_2, \dotsc, D_{d+1}$ are the $T$-invariant divisors on $\PP^d$.
  Using the notation from Example~\ref{exa:ampleTangent}, let $C_{i,j}$ denote
  the $T$-invariant curve corresponding to the cone
  $\tau_{i,j} = \sigma_i \cap \sigma_j \in \Sigma(d-1)$ where
  $1 \leq i < j \leq d+1$.  Restricting the short exact sequence to the curve
  $C_{i,j}$, we obtain
  $\cQ|_{C_{i,j}} \cong \cO_{\PP^1}(a_i + a_j) \oplus \bigl( \bigoplus_{k = 1,
    k \neq i,j}^{d+1} \cO_{\PP^1}(a_k) \bigr)$.
  
  If $\cE = \cQ(\ell)$ and $\cE$ is nef, then we have $a_k + \ell \geq 0$ for
  all $1 \leq k \leq d+1$ which means that the vector bundle
  $\mathcal{S} := \bigoplus_{k=1}^{d+1} \cO_{\PP^d}(a_k + \ell)$ is globally
  generated.  Since $\cE$ is a quotient of $\mathcal{S}$, we conclude that
  $\cE$ is also globally generated; see Example~6.1.4 in \cite{PAG2}.  If
  $\cE = \cQ^*(\ell)$ and $\cE$ is nef, then we have $\ell - a_k \geq 0$ for
  all $1 \leq k \leq d+1$ and $\ell - a_i - a_j \geq 0$ for all
  $1 \leq i < j \leq d+1$.  The functorial properties of the dual imply that
  $\cQ^*(\ell) \hookrightarrow \bigoplus_{k = 1}^{d+1} \cO_{\PP^d}(\ell -
  a_k)$
  and
  $\cQ^*(\ell) \cong \bigl( \bigwedge^{d-1} \cQ^*(\ell) \bigr)^* \otimes \det
  \bigl( \cQ^*(\ell) \bigr)$.
  It follows that $\cE$ is a quotient of the vector bundle
  $\mathcal{S}' := \bigl( \bigwedge^{d-1} \bigl( \bigoplus_{k=1}^{d+1}
  \cO_{\PP^d}(\ell-a_k) \bigr) \bigr)^* \otimes \det \bigl( \cQ^*(\ell)
  \bigr)$.
  Since
  $\bigwedge^{d-1} \bigl( \bigoplus_{k=1}^{d+1} \cO_{\PP^d}(\ell-a_k) \bigr)
  \cong \bigoplus_{1 \leq k_1 < k_2 < \dotsb < k_{d-1} \leq d+1}
  \cO_{\PP^d}\bigl( (d-1)\ell - a_{k_1} - a_{k_2} - \dotsb - a_{k_{d-1}}
  \bigr)$
  and
  $\det \bigl( \cQ^*(\ell) \bigr) \cong \cO_{\PP^d}( d \ell - a_1 - a_2 -
  \dotsb - a_{d+1})$,
  we see that $\mathcal{S}'$ is a direct sum of line bundles of the form
  $\cO_{\PP^d}(\ell - a_j - a_k)$ which implies that both $\mathcal{S}'$ and
  $\cE$ are globally generated.
\end{proof}

To complement Examples~4.9--4.10 in \cite{HMP}, we end this section by
illustrating that the higher cohomology groups of a globally-generated ample
toric vector bundle on a smooth toric variety may be nonzero.

\begin{example} 
  \label{exa:nonvanishing2}
  Consider the globally-generated toric vector bundle $\cG$ appearing in
  Example~\ref{exa:nonvanishing}.  Restricting to the $T$-invariant curves
  gives $\cG|_{C_1} \cong \cO_{\PP^1}(5) \oplus \cO_{\PP^1}(2)$,
  $\cG|_{C_2} \cong \cO_{\PP^1}(1) \oplus \cO_{\PP^1}(1)$,
  $\cG|_{C_3} \cong \cO_{\PP^1}(6) \oplus \cO_{\PP^1}(1)$,
  $\cG|_{C_4} \cong \cO_{\PP^1}(8) \oplus \cO_{\PP^1}(1)$, and shows that
  $\cG$ is ample.  Furthermore, Theorem~4.2.1 in \cite{Kly} establishes that
  the $T$-equivariant Euler characteristic of $\cG$ is
  \begin{align*}
    \chi(\cG) &= \sum_i (-1)^i \dim H^i(X, \cG)_\uu \cdot t^{\uu} =
    \tfrac{t_1^{-2}t_2^{2} + t_1^{4}t_2^3}{(1 - t_1^{})(1 - t_2^{})} +
    \tfrac{t_1^{-3}t_2^{2} + t_1^{3}t_2^3}{(1 - t_1^{})(1 - t_1^{-1}t_2^{-1})}
    + \tfrac{t_1^{-4}t_2^{} + t_1^{-3}t_2^{-3}}{(1 - t_1^{})(1 -t_1^{}t_2^{})}
    + \tfrac{t_1^{-2}t_2^{-3} + t_1^{4}t_2^{}}{(1 - t_1^{-1})(1 - t_2^{})} \\
    &= t_1^{4}t_2^{3} + t_1^{4}t_2^{2} + t_1^{4}t_2^{} + t_1^{3}t_2^{3} +
    t_1^{3}t_2^{2} + t_1^{3}t_2^{} + t_1^{2}t_2^{2} + t_1^{2}t_2^{} +
    t_1^{}t_2^{} \\
    &\relphantom{==}- t_1^{-1} + t_1^{-2}t_2^{2} + t_1^{-2}t_2^{} +
    t_1^{-2}t_2^{-2} + t_1^{-2}t_2^{-3} + t_1^{-3}t_2^{2} + t_1^{-3}t_2^{} +
    t_1^{-3}t_2^{-3} + t_1^{-4}t_2^{} \, ,
  \end{align*}
  so we have $H^1(X, \cG)_{(-1,0)} \neq 0$.  Using Theorem~4.1.1 in
  \cite{Kly}, a longer calculation confirms that $H^1(X,\cG)_{\uu} \cong \CC$
  when $\uu = (-1,0)$ and $H^1(X,\cG)_{\uu} = 0$ when $\uu \neq (-1,0)$.  In
  Figure~\ref{fig:nonvanishing}, the triangle represents the unique character
  for which the higher cohomology groups do not vanish. \hfill $\Diamond$
\end{example}

\begin{remark}
  \label{rem:nonvanishing}
  Using the techniques from Example~\ref{exa:nonvanishing2} or Example~4.3.5
  in \cite{Kly}, we see that $H^1(\PP^2, \cF)_\uu \neq 0$ where $\uu = (1,-1)$
  and $\cF$ is the toric vector bundle appearing in Example~\ref{exa:notBPF}.
  In Figure~\ref{fig:notSpanned}, the triangle represents the unique character
  for which the higher cohomology groups do not vanish.
\end{remark}

\section{Higher-Order Jets}
\label{sec:jets}

\noindent
This final section relates positivity of higher-order jets to properties of
the associated parliament of polytopes.  In particular, we determine which
results for jets of line bundles on smooth toric varieties extend to
higher-rank toric vector bundles.  For toric vector bundles, we also provide
an explicit polyhedral characterization for very ampleness.

Fix $\ell \in \NN$.  A vector bundle $\cE$ \emph{separates $\ell$-jets} if, for
every closed point $x \in X$ with maximal ideal
$\mathfrak{m}_x \subseteq \cO_X$, the map
$J^\ell_x \colon H^0(X, \cE) \to H^0(X, \cE \otimes_{\cO_{X}}
\cO_X/\mathfrak{m}_x^{\ell+1})$, which evaluates a global section and its
derivatives of order at most $\ell$ at $x$, is surjective; compare with
Definition~5.1.15 in \cite{PAGI}.  When $X$ is a toric variety, this map is
$T$-equivariant, because differentiation is $\CC$-linear.  As a special case,
we see that a vector bundle separates $0$-jets if and only if it is globally
generated.  A vector bundle that separates $\ell$-jets is also called
\emph{$\ell$-jet spanned}.

As a stronger attribute, we say that a vector bundle $\cE$ is \emph{$\ell$-jet
  ample} if, for all distinct closed points $x_1, x_2, \dotsc, x_t \in X$ and
for all positive integers $\ell_1, \ell_2, \dotsc, \ell_t$ satisfying
$\sum_{i=1}^t \ell_i = \ell+1$, the natural map
$\psi \colon H^0(X, \cE) \to H^0 \bigl( X, \cE \otimes_{\cO_{X}}
\cO_X/(\mathfrak{m}_{x_1}^{\ell_1} \cdot \mathfrak{m}_{x_2}^{\ell_2} \dotsb
\mathfrak{m}_{x_t}^{\ell_t}) \bigr) = \bigoplus_{i=1}^t H^0 (X, \cE
\otimes_{\cO_X} \cO_X / \mathfrak{m}_{x_i}^{\ell_i})$ is surjective.  Hence, a
$\ell$-jet ample vector bundle does separate $\ell$-jets, and a vector bundle
separates $0$-jets if and only if it is $0$-jet ample.  Proposition~4.2 in
\cite{BDRS} proves that every $1$-jet ample vector bundle on a smooth
projective variety is very ample, and Example~4.3 in \cite{BDRS} shows that
the converse does not always hold.  If $0 \leq m \leq \ell$, then a vector
bundle that separates $\ell$-jets also separates $m$-jets, and a vector bundle
that is $\ell$-jet ample is also $m$-jet ample.

We start by placing ampleness into this hierarchy of positivity properties on
a smooth toric variety.

\begin{lemma}
  \label{lem:ample}
  Every toric vector bundle that separates $1$-jets is ample.
\end{lemma}

\begin{proof}  
  Let $\cE$ be a toric vector bundle that separates $1$-jets.  For any
  $T$-invariant curve $C$, the restriction $\cE|_C$ separates $1$-jets and
  splits $T$-equivariantly into sum of line bundles.  For a line bundle on a
  toric variety, Theorem~4.2 in \cite{DiRocco} shows that separating $1$-jets
  is equivalent to being ample.  Hence, if
  $\cE|_C \cong \cO_{\PP^1}(a_1) \oplus \cO_{\PP^1}(a_2) \oplus \cdots \oplus
  \cO_{\PP^1}(a_r)$, then each line bundle $\cO_{\PP^1}(a_i)$ is ample.
  Therefore, the restriction to every $T$-invariant curve is ample, which
  ensures that $\cE$ is ample; see Theorem~2.1 in \cite{HMP}.
\end{proof}

We next characterize the toric vector bundles that separate $\ell$-jets by
enhancing Theorem~\ref{thm:main1}.

\begin{theorem}
  \label{thm:kjet}
  A toric vector bundle $\cE$ separates $\ell$-jets, for $\ell \geq 1$, if and only
  if, for all maximal cones $\sigma \in \Sigma(d)$, the following conditions
  hold:
  \begin{enumerate}[\upshape (i)]
  \item the associated characters $\bm{u}(\sigma)$ are vertices of polytopes
    in the parliament for $\cE$,
  \item the edges adjacent to these vertices correspond to the generators of
    the dual cone $\sigma^\vee$,
  \item the edges adjacent to these vertices have normalized length at least
    $\ell$,  and
  \item the elements indexing these polytopes form a basis in the matroid
    $\Mat(\cE)$.
  \end{enumerate}
\end{theorem}

\noindent 
If we ignore the conditions on the edges, then we recover
Theorem~\ref{thm:main1} which characterizes toric vector bundles that separate
$0$-jets.

\begin{proof}
  The locus in the toric variety $X$, on which
  $H^0(X,\cE) \to H^0(X, \cE \otimes_{\cO_X} \cO_X / \mathfrak{m}_x^{\ell+1})$
  is not surjective, is closed and $T$-invariant.  Since $X$ is complete, it
  follows that $\cE$ separates $\ell$-jets if and only if it separates
  $\ell$-jets at the $T$-fixed points.

  Fix a maximal cone
  $\sigma = \pos(\vv_1, \vv_2, \dotsc, \vv_d) \in \Sigma(d)$, let
  $\underline{\bm{u}}(\sigma) = \{ \uu_{\sigma,1}, \uu_{\sigma,2}, \dotsc,
  \uu_{\sigma,s} \}$ be the underlying set of associated characters, and let
  $\sB_{\sigma} = \sB_{\sigma,1} \sqcup \sB_{\sigma,2} \sqcup \dotsb \sqcup
  \sB_{\sigma,s}$ be the corresponding partition of a compatible basis in the
  matroid $\Mat(\cE)$; see Section~\ref{sec:basepointfree}. The vector bundle
  $\cE$ separates $\ell$-jets at the $T$-fixed point $x_{\sigma}$ if and only if
  the natural map
  \[
    J^\ell_{x_\sigma} \colon H^0(X,\cE) \to H^0( X, \cE \otimes
    \cO_X/\mathfrak{m}_{x_{\sigma}}^{\ell+1}) \cong \Span(\sB_{\sigma})
    \otimes_{\CC} \CC^{\binom{\ell+d}{d}}
  \]
  is surjective, where the standard basis for the vector space
  $\CC^{\binom{\ell+d}{d}}$ corresponds to the partial derivatives of order
  less than $\ell$. Since a $T$-equivariant global section
  $\ee' \otimes \chi^{-\uu}$ is given in local coordinates near $x_\sigma$ by
  \eqref{eq:local}, the map $J^\ell_{x_{\sigma}}$ sends
  $\ee' \otimes \chi^{-\uu}$ to the first $\ell$ terms of the Taylor expansion
  about $x_{\sigma}$.  Hence, for
  $\mathbf{m} = (m_1, m_2, \dotsc, m_d) \in \NN^d$ satisfying
  $m_1 + m_2 + \dotsb + m_d \leq \ell$, the $\mathbf{m}$-th component of
  $J^\ell_{x_{\sigma}}(\ee' \otimes \chi^{- \uu})$ is given in local
  coordinates by
  \[
    \sum_{\ee \in \sB_\sigma} \lambda_{\ee} \Biggl( \ee \otimes \frac{1}{m_1!
      m_2! \dotsb m_d!}  \frac{\partial^{m_1 + m_2 + \dotsb +
        m_d}}{\partial^{m_1}y_1 \partial^{m_2}y_2 \dotsb \partial^{m_d}y_d}
    \biggl( \prod_{i = 1}^{d} y_i^{- \langle \uu, \vv_i \rangle +
      \varphi_{\ell, \sigma}(\vv_i)} \biggr)\Biggr) \Bigg|_{y_1 = y_2 = \dotsb
      = y_d = 0} \, .
  \]
  The $\ee$-th summand in this expression has neither a zero nor a pole at
  $(y_1, y_2, \dotsc, y_d) = (0,0, \dotsc, 0)$ if and only if we have
  $- \langle \uu, \vv_i \rangle + \varphi_{\ee}(\vv_i) = m_i$ for all
  $1 \leq i \leq d$.  By Lemma~\ref{lem:function}, it follows that there
  exists an index $k$ such that $\ee \in \sB_{\sigma, k}$ and
  $\uu = \uu_{\sigma,k} + \mathbf{m}$.  In this case, the lattice point
  $\uu_{\sigma,k} - \sum_{i=1}^{d} m_i \ww_i$, where
  $\ww_1, \ww_2, \dotsc, \ww_d$ are the unique minimal generators of the dual
  cone $\sigma^\vee$, belongs to the polytope $P_\ee$ in the parliament for
  $\cE$.  Hence, the $\mathbf{m}$-th component of
  $J^\ell_{x_{\sigma}}(\ee \otimes \chi^{- \uu})$ is nonzero if and only if
  the global section includes summands of the form
  $\sum_{\ee \in \sB_{\sigma, k}} \lambda_{\ee} (\ee \otimes
  \chi^{-\uu_{\sigma, k}- \mathbf{m}})$, which map to
  $\sum_{\ee \in \sB_{\sigma, k}} \lambda_\ee \, \ee$.  For the map
  $J^\ell_{x_{\sigma}}$ to be surjective, we need each vector
  $\ee \in \sB_\sigma$ to appear in each component of factor
  $\CC^{\binom{\ell+d}{d}}$.  Therefore, the map $J^\ell_{x_{\sigma}}$ is
  surjective if and only if there exists a compatible basis $\sB_\sigma$ such
  that each
  $\ee \otimes \chi^{-\uu_{\sigma,k} - m_1 \ww_1 - m_2 \ww_2 - \dotsb - m_d
    \ww_d}$, for $1 \leq k \leq s$, $\ee \in \sB_{\sigma,k}$, and
  $\mathbf{m} \in \NN^d$ satisfying $m_1 + m_2 + \dotsb + m_d \leq \ell$, is a
  global section.  By convexity, this characterization is equivalent to
  requiring that the edges through the vertex $\uu_{\sigma,k}$ in the
  directions of dual vectors $\ww_i$ have normalized length at least $\ell$.
\end{proof}

With Theorem~\ref{thm:kjet}, we easily verify that the tangent bundle on
projective space separates $1$-jets.
    
\begin{example}
  As computed in Example~\ref{exa:tangentBundle}, the parliament of polytopes
  for the tangent bundle $\cT_{\PP^d}$ consists of
  $P_{\vv_i} = \conv( \mathbf{0}, \ww_i \!-\! \ww_1, \ww_i \!-\! \ww_2,
  \dotsc, \ww_i \!-\! \ww_{i-1}, \ww_i, \ww_i \!-\! \ww_{i+1}, \ww_i \!-\!
  \ww_{i+1}, \dotsc, \ww_i \!-\!  \ww_d)$ for $1 \leq i \leq d$, and
  $P_{\vv_{d+1}} = \conv( \mathbf{0}, - \ww_1, - \ww_2, \dotsc, - \ww_d)$.
  Hence, the associated characters are vertices of polytopes in the
  parliament, the edges in each polytope have normalized length $1$ and point
  in directions corresponding to generators of the dual cone, and the elements
  indexing these polytopes equal the unique choice of compatible basis.
  Therefore, the tangent bundle $\cT_{\PP^d}$ separates $1$-jets. \hfill
  $\Diamond$
\end{example}

Since Example~\ref{exa:notBPF} exhibits an ample toric vector bundle that is
not globally generated, the converse to Lemma~\ref{lem:ample} is false.  To
sharpen this distinction, we present an ample toric vector bundle that is
globally generated but does not separate $1$-jets.
  
\begin{example}
  \label{exa:span}
  Using the notation from Example~\ref{exa:notBPF} and
  Example~\ref{exa:notBPF2}, consider the toric vector bundle $\cH$ of rank
  $3$ on $\PP^2$ defined by the following decreasing filtrations:
  \begin{xalignat*}{2}
    E^{1}(j) &= 
    \left\{
      \renewcommand{\arraystretch}{0.9}
      \renewcommand{\arraycolsep}{1pt}
      \begin{array}{lcrcl}
        E & \;\; \text{if} \;\;  &   &j & \leq -2  \\
        \Span ( \ee_1, \ee_2 ) & \text{if} & -2 < & j & \leq -1 \\
        \Span ( \ee_1 ) & \text{if} & -1 < & j &\leq 2 \\
        0 & \text{if} &  2 < & j & 
      \end{array}
    \right. &  
    E^{3}(j) &= 
    \left\{
      \renewcommand{\arraystretch}{0.9}
      \renewcommand{\arraycolsep}{1pt}
      \begin{array}{lcrcl}
        E & \;\; \text{if} \;\; &  & j & \leq 1 \\
        \Span ( \ee_3-\ee_2, \ee_1 - \ee_2 ) & \text{if} & 1 < & j &
        \leq 3 \\
        \Span ( \ee_1 -\ee_2 ) & \text{if} & 3 < & j & \leq 4 \\
        0 & \text{if} & 4 < & j 
      \end{array}
    \right. \\
    E^{2}(j) &= 
    \left\{
      \renewcommand{\arraystretch}{0.9}
      \renewcommand{\arraycolsep}{1pt}
      \begin{array}{lcrcl}
        E & \;\; \text{if} \;\; & & j &\leq -2 \\
        \Span ( \ee_2, \ee_3 ) & \text{if} & -2 < &j &\leq 0 \\
        \Span ( \ee_3 ) & \text{if} & 0 < & j & \leq 2 \\
        0 & \text{if} & 2 < & j & \, .
      \end{array}
    \right.
  \end{xalignat*}
  It follows that the ground set of the matroid $\Mat(\cH)$ is
  $\{ \ee_1, \ee_1 - \ee_2, \ee_2, \ee_2 - \ee_3, \ee_3 \}$. On each maximal
  cone, the associated characters and the unique choice of compatible bases
  are
  \begin{xalignat*}{2}
    \bm{u}(\sigma_1) &= \{ (-2,-2), (-3,0), (-3,2) \} \, , & 
    \sB_{\sigma_1} &= \{ \ee_1-\ee_2, \ee_2-\ee_3, \ee_3 \} \, ,\\
    \bm{u}(\sigma_2) &= \{ (2,-3), (-1,-3), (-2,-1) \} \, , & 
    \sB_{\sigma_2} &= \{\ee_1, \ee_1-\ee_2, \ee_2-\ee_3 \} \, , \\
    \bm{u}(\sigma_3) &= \{ (2,-2), (-1,0), (-2,2) \} \, , &
    \sB_{\sigma_3} &= \{ \ee_1, \ee_2,  \ee_3 \} \, , 
  \end{xalignat*}
  so the convex polytopes in the parliament for $\cH$ are
  \begin{xalignat*}{2}
    P_{\ee_1} &= \conv \bigl( (1,-2), (2,-2), (2,-3) \bigr) \, , &
    P_{\ee_2 - \ee_3} &= \conv \bigl( (-3,0), (-2,0), (-2,-1) \bigr) \, ,\\
    P_{\ee_1 - \ee_2} &= \conv \bigl( (-2,-2), (-1,-2), (-1,-3) \bigr) \, , &
    P_{\ee_3} &= \conv \bigl( (-3,2), (-2,2), (-2,1) \bigr) \, , \\
    P_{\ee_2} &= \conv \bigl( (-1,0) \bigr) \, .
  \end{xalignat*}
  In Figure~\ref{fig:span}, the associated characters are represented by
  \begin{figure}[t]
    \begin{tikzpicture}[x=0.75cm, y=0.75cm, line width=1.25pt]
      \foreach \x in {-4,-3,-2,-1,1,2,3,4} {
        \draw[color=white!40!black] (\x, 2pt) -- (\x, -2pt);}
      \foreach \y in {-3,-2,-1,1,2} {
        \draw[color=white!40!black] (2pt, \y) -- (-2pt, \y);}
      \draw[color=white!40!black, <->] (-4.4, 0) -- (4.4,0) {};
      \draw[color=white!40!black, <->] (0, -3.4) -- (0, 2.4) {};
      \draw[color=blue, fill=blue, opacity=0.5] (2,-2) -- (1,-2) -- (2,-3) --
      (2,-2) -- cycle;
      \draw[color=blue, fill=blue, opacity=0.5] (-2,-2) -- (-1,-3) -- (-1,-2)
      -- (-2,-2) -- cycle;
      \draw[color=blue, fill=blue, opacity=0.5] (-3,2) -- (-2,2) -- (-2,1) --
      (-3,2) -- cycle;
      \draw[color=blue, fill=blue, opacity=0.5] (-3,0) -- (-2,0) -- (-2,-1) --
      (-3,0) -- cycle;
      \draw[color=blue] (2,-2) -- (-2,-2) -- (-1,-3) -- (-1,0) -- (-3,0) --
      (-2,-1) -- (-2,2) -- (-3,2) -- (2,-3) -- (2,-2) -- cycle;
      \node[circle, fill=black, inner sep=2.0pt] () at (-2,1) {};
      \node[circle, fill=black, inner sep=2.0pt] () at (-2,0) {};
      \node[circle, fill=black, inner sep=2.0pt] () at (-1,-2) {};
      \node[circle, fill=black, inner sep=2.0pt] () at (1,-2) {};
      \node[rectangle, fill=black, inner sep=2.5pt] () at (2,-2) {};
      \node[rectangle, fill=black, inner sep=2.5pt] () at (-1,0) {};
      \node[rectangle, fill=black, inner sep=2.5pt] () at (-2,2) {};
      \node[star, star points=5, star point ratio = 0.3, fill=black, 
      inner sep=3.5pt] () at (-2,-2) {};
      \node[star, star points=5, star point ratio = 0.3, fill=black, 
      inner sep=3.5pt] () at (-3,0) {};
      \node[star, star points=5, star point ratio = 0.3, fill=black, 
      inner sep=3.5pt] () at (-3,2) {};
      \node[diamond, fill=black, inner sep=2.0pt] () at (2,-3) {};
      \node[diamond, fill=black, inner sep=2.0pt] () at (-2,-1) {};
      \node[diamond, fill=black, inner sep=2.0pt] () at (-1,-3) {};
      \node[color=black] () at (2.5,-2.6) {$P_{\ee_1}$};
      \node[color=black] () at (-2.5,2.4) {$P_{\ee_3}$};
      \node[color=black] () at (-3,-0.6) {$P_{\ee_3 - \ee_2}$};
      \node[color=black] () at (-2,-2.6) {$P_{\ee_1 - \ee_2}$};
      \node[color=black] () at (-0.8,0.5) {$P_{\ee_2}$};
    \end{tikzpicture}
    \caption{The parliament of polytopes for $\cH$}
    \label{fig:span}
  \end{figure}
  asterisks, diamonds, and squares respectively.  The polytopes are
  represented by shaded regions and the other lattices points lying in the
  polytopes are represented by circles.  Using Theorem~\ref{thm:main1}, we see
  that $\cH$ is globally generated.  In contrast, Theorem~\ref{thm:kjet}
  implies that $\cH$ does not separate $1$-jets because $P_{\ee_2}$ is simply
  a point.  Lastly, restricting to the $T$-invariant curves gives
  $\cH|_{C_1} \cong \cO_{\PP^1}(3) \oplus \cO_{\PP^1}(3) \oplus
  \cO_{\PP^1}(1)$,
  $\cH|_{C_2} \cong \cO_{\PP^1}(4) \oplus \cO_{\PP^1}(2) \oplus
  \cO_{\PP^1}(1)$, and
  $\cH|_{C_3} \cong \cO_{\PP^1}(5) \oplus \cO_{\PP^1}(1) \oplus
  \cO_{\PP^1}(1)$, so the toric vector bundle is ample.  \hfill $\Diamond$
\end{example}

On a smooth projective variety, being $1$-jet ample is generally a stronger
condition than separating $1$-jets, as Example~2.3 in \cite{LM} and
Example~4.6 in \cite{Langer} demonstrate for line bundles.  For line bundles
on a smooth complete toric variety, these conditions are equivalent; see
\cite{DiRocco}.  Extending this result, we prove that these conditions are
equivalent for toric vector bundles on a smooth complete toric variety.

\begin{theorem}
  \label{thm:sepIffample}
  A toric vector bundle separates $\ell$-jets if and only if it is $\ell$-jet ample.
\end{theorem}

\begin{proof} 
  It suffices to show that every toric vector bundle $\cE$ which separates
  $\ell$-jets is $\ell$-jet ample. The locus in the toric variety
  $\prod_{i=1}^t X$, on which
  $H^0(X,\cE) \to \bigoplus_{i=1}^t H^0(X, \cE \otimes_{\cO_X} \cO_X /
  \mathfrak{m}_{x_{i}}^{\ell_i})$ is not surjective, is closed and
  $T$-invariant.  Since $X$ is complete, it follows that $\cE$ is $\ell$-jet
  ample if and only if it is $\ell$-jet ample at the $T$-fixed points.  Thus,
  it is enough to prove that, for all distinct $T$-fixed points
  $x_{\sigma_1}, x_{\sigma_2}, \dotsc, x_{\sigma_t}$ and all positive integers
  $\ell_1, \ell_2, \dotsc, \ell_t$ satisfying
  $\sum_{i=1}^{t} \ell_i = \ell+1$, the map
  $\psi \colon H^0(X, \cE) \to \bigoplus_{i=1}^t H^0(X, \cE \otimes_{\cO_X}
  \cO_X / \mathfrak{m}_{x_{\sigma_i}}^{\ell_i})$ is surjective.
  
  Since $\cE$ separates $\ell_i$-jets, for all $\ell_i\leq \ell,$ the map
  $\psi$ surjects onto each individual summand.  Consider a $T$-equivariant
  global section $\ee \otimes \chi^{-\uu}$ where the element $\ee$ belongs to
  the ground set of the matroid $\Mat(\cE)$ and
  $0 \neq J_{x_{\sigma_1}}^{\ell_1-1}(\ee \otimes \chi^{-\uu}) \in H^0(X, \cE
  \otimes_{\cO_X} \cO_X / \mathfrak{m}_{x_{\sigma_1}}^{\ell_1})$.  To prove
  that $\psi$ is surjective, it is enough to show that
  $J_{x_{\sigma_2}}^{\ell_2-1}(\ee \otimes \chi^{-\uu}) = 0$ because we may
  reindex the $T$-fixed points.  As in the proof of Theorem~\ref{thm:kjet},
  the hypothesis $J_{x_{\sigma_1}}^{\ell_1-1}(\ee \otimes \chi^{-\uu}) \neq 0$
  implies that the global section $\ee \otimes \chi^{-\uu}$ corresponds to the
  lattice $\uu \in P_\ee$ and the lattice distance from the vertex
  $\uu_{\sigma_1}$ of $P_\ee$ associated to the maximal cone $\sigma_1$ is at
  most $\ell_1-1$.  Similarly, if
  $J_{x_{\sigma_2}}^{\ell_2-1}(\ee \otimes \chi^{-\uu}) \neq 0$, then the
  lattice distance from $\uu \in P_\ee$ to the vertex $\uu_{\sigma_2}$ of
  $P_\ee$ associated to the maximal cone $\sigma_2$ would also be at most
  $\ell_2 - 1$.  As $\cE$ separates $\ell$-jets at $x_{\sigma_1}$,
  Theorem~\ref{thm:kjet} implies that the lattice length of each edge in
  $P_{\ee}$ emanating from the vertex corresponding to $\sigma_1$ is at least
  $\ell$.  Since $\ell_1 + \ell_2 -2 \leq \ell -1$, the convexity of $P_\ee$
  guarantees that the lattice point $\uu$ cannot be simultaneously this close
  to both $\uu_{\sigma_1}$ and $\uu_{\sigma_2}$.  Therefore, we conclude that
  $J_{x_{\sigma_2}}^{\ell_2-1}(\ee \otimes \chi^{-\uu}) = 0$ and $\psi$ is
  surjective.
\end{proof}

For a line bundle on a smooth toric variety, Theorem~4.2 in \cite{DiRocco}
establishes that separating $1$-jets is equivalent to being very ample.  As a
final result, we generalize this equivalence to higher-rank toric vector
bundles on a smooth toric variety.

\begin{theorem}
  \label{thm:sepIffvery}
  A toric vector bundle separates $1$-jets if and only if it is very ample.
\end{theorem}

\begin{proof}  
  It suffices to show that every very ample toric vector bundle $\cE$
  separates $1$-jets at the $T$-fixed points.  Let $X$ be the smooth toric
  variety determined by the fan $\Sigma$.  Fix a maximal cone
  $\sigma_0 \in \Sigma(d)$ and consider the blowing up $\pi \colon X' \to X$
  of $X$ at $x_{\sigma_0}$ with exceptional divisor
  $D_0 := \pi^{-1}(x_{\sigma_0})$.  Since $\cE$ is very ample, Corollary~1 in
  \cite{BSS} establishes that the toric vector bundle
  $\cE' := \pi^*(\cE) \otimes_{\cO_{X'}} \cO_{X'}(-D_0)$ is globally
  generated.

  To complete the proof, we relate the parliament of polytopes for $\cE'$ and
  $\cE$.  First, we describe the underlying fan for $X'$.  Let
  $\vv_1, \vv_2, \dotsc, \vv_n$ be the primitive lattice vectors generating
  the rays in $\Sigma$.  By reordering these rays if necessary, we may assume
  that $\sigma_0 = \pos(\vv_1, \vv_2, \dotsc, \vv_d)$.  If
  $\vv_0 := \vv_1 + \vv_2 + \dotsb + \vv_d$ and $\Sigma'$ is the fan of $X'$,
  then the primitive lattice vectors generating the rays in $\Sigma'$ are
  $\vv_0, \vv_1, \dotsc, \vv_n$, and the maximal cones are
  $\Sigma'(d) = \bigl( \Sigma(d) \setminus \sigma_0 \bigr) \cup \{ \sigma_1,
  \sigma_2, \dotsc, \sigma_d \}$
  where
  $\sigma_i := \pos(\vv_0, \vv_1, \dotsc, \vv_{i-1}, \vv_{i+1}, \vv_{i+2},
  \dotsc, \vv_d)$
  for $1 \leq i \leq d$; compare with Example~3.1.15 in \cite{CLS}.

  We next specify the linear invariants which determine the toric vector
  bundle $\cE'' := \pi^*(\cE)$ on $X'$.  The characters associated to $\cE''$
  are $\bm{u}_{\cE''}(\sigma') = \bm{u}_{\cE}(\sigma')$ for all
  $\sigma' \in \Sigma(d) \setminus \sigma_0$ and
  $\bm{u}_{\cE''}(\sigma_i) = \bm{u}_{\cE}(\sigma_0)$ for all
  $1 \leq i \leq d$.  The compatible decreasing filtrations corresponding to
  $\cE''$ are identical to those for $\cE$ along the rays generated by $\vv_i$
  for $1 \leq i \leq n$.  Along the new ray generated by $\vv_0$, the
  decreasing filtration for $\cE''$ is
  ${E''}^{\vv_0}(j) = \sum_{\langle \uu, \vv_0 \rangle \geq j}
  E_{\uu}^{\sigma_0}$, where $E_{\uu}^{\sigma_0}$ is the linear subspace
  associated to $\cE$; see Section~\ref{sec:bundles}.  It follows that
  $\sB(\cE'')_{\sigma_i} = \sB(\cE)_{\sigma_0}$ for all $1 \leq i \leq d$.
  
  Finally, to analyze $\cE'$, set $\cL := \cO_{X'}(-D_0)$.  Writing
  $\ww_1, \ww_2, \dotsc, \ww_d \in M$ for the minimal generators the dual cone
  $\sigma^\vee$, Example~\ref{exa:lineBundle} shows that the characters
  associated to the line bundle $\cL$ are
  $\bm{u}_{\!\cL}(\sigma') = \{ \mathbf{0} \}$ for all
  $\sigma' \in \Sigma(d) \setminus \sigma_0$ and
  $\bm{u}_{\! \cL}(\sigma_i) = \{ \ww_i \}$ for all $1 \leq i \leq d$.
  Combining this data with that for $\cE''$, we see that the characters
  associated to toric vector bundle $\cE' = \cE'' \otimes_{\cO_X} \cL$ are
  $\bm{u}_{\cE'}(\sigma') = \bm{u}_{\cE}(\sigma')$ for all
  $\sigma' \in \Sigma(d) \setminus \sigma_0$ and
  $\bm{u}_{\cE'}(\sigma_i) = \{ \uu + \ww_i : \uu \in \bm{u}_{\cE}(\sigma_0)
  \}$
  for all $1 \leq i \leq d$.  As $\cL$ is a line bundle, we also have
  $\sB(\cE')_{\sigma_i} = \sB(\cE'')_{\sigma_i} = \sB(\cE)_{\sigma_0}$ for all
  $1 \leq i \leq d$.  If $\uu_\ell \in \bm{u}_{\cE}(\sigma_0)$ corresponds to
  the vector $\ee_{\ell,\sigma_0} \in E$, then the element
  $\uu_\ell + \ww_i \in \bm{u}_{\cE'}(\sigma_i)$ corresponds to the same
  vector in $E$.  Since $\cE'$ is globally generated and $\cE$ is very ample,
  Theorem~\ref{thm:main1} implies that
  $\uu_\ell + \ww_i \in P_{\ee_{\ell,\sigma_0}}$ and
  $\uu_\ell \in P_{\ell,\sigma_0}$ for all
  $\uu_\ell \in \bm{u}_{\cE}(\sigma_0)$ and for all $1 \leq i \leq d$.
  Applying Theorem~\ref{thm:kjet}, we conclude that $\cE$ separates $1$-jets.
\end{proof} 

An a consequence, we have a polyhedral characterization for very ample toric
vector bundles.

\begin{corollary}
  \label{cor:very}
  A toric vector bundle $\cE$ is very ample if and only if, for all
  maximal cones $\sigma \in \Sigma(d)$, the following conditions hold:
  \begin{enumerate}[\upshape (i)]
  \item the associated characters $\bm{u}(\sigma)$ are vertices of polytopes
    in the parliament for $\cE$,
  \item the edges adjacent to these vertices correspond to the generators of
    the dual cone $\sigma^\vee$,
  \item the edges adjacent to these vertices have lattice length at least $1$, and
  \item the elements indexing these polytopes form a basis in the matroid
    $\Mat(\cE)$.
  \end{enumerate}
\end{corollary}

\begin{proof}
  Combine Theorem~\ref{thm:kjet} and Theorem~\ref{thm:sepIffvery}.
\end{proof}

\begin{proof}[Proof of Theorem~\ref{thm:main?}]
  This follows immediately by combining Theorem~\ref{thm:sepIffample} and
  Theorem~\ref{thm:sepIffvery}.
\end{proof}

\begin{remark}
  \label{rem:lowRank}
  Combining Example~\ref{exa:span} with Theorem~\ref{thm:sepIffvery}, we see
  that the toric vector bundle $\cH$ is globally generated and ample but not
  very ample, answering the first part of Question~7.5 in \cite{HMP}.
  Modifying the proof of Proposition~\ref{pro:lowRank} by replacing some
  non-strict inequalities with strict inequalities, we also obtain a partial
  converse to Lemma~\ref{lem:ample}: if $\cE$ is a toric vector bundle on
  $\PP^d$ with rank at most $d$, then $\cE$ is ample if and only if it
  separates $1$-jets.  Hence, $\cH$ has minimal rank among all
  globally-generated ample toric vector bundles on $\PP^2$ that are not very
  ample.
\end{remark}


\raggedright

\end{document}